\input ./style/arxiv-general.cfg
\documentclass[aop,MSNbibl,seceqn,dvips]{arximspdf}
\makeatletter
   \@ifpackageloaded{graphicx}{}{\usepackage{graphicx}}
\makeatother

\doi{10.1214/13-AOP897} 
\volume{43}
\issue{3}
\pubyear{2015}
\firstpage{1350}
\lastpage{1398}

\makeatletter
\newcommand{\rrVert}{\Vert}
\newcommand{\rrvert}{\vert}
\newcommand{\llVert}{\Vert}
\newcommand{\llvert}{\vert}
\newtheorem{lemma}{Lemma}[section]
\newtheorem{proposition}{Proposition}[section]
\newtheorem{theorem}{Theorem}[section]
\newtheorem{corollary}{Corollary}[section]
\newproclaim{definition}{Definition}[section]
\newproclaim{example}{Example}[section]
\newproclaim{remark}{Remark}[section]
\newproclaim{assumption}{Assumption}[section]
\makeatother

\begin{document}
\begin{frontmatter}

\title{Pointwise eigenfunction estimates and intrinsic
ultracontractivity-type properties of Feynman--Kac semigroups for a
class of L\'evy processes}
\runtitle{Eigenfunction estimates and intrinsic ultracontractivity\hspace*{10pt}}

\begin{aug}
\author[A]{\fnms{Kamil} \snm{Kaleta}\corref{}\ead[label=e1]{kamil.kaleta@pwr.wroc.pl}\thanksref{T1}}
\and
\author[B]{\fnms{J\'ozsef} \snm{L{\H o}rinczi}\ead[label=e2]{J.Lorinczi@lboro.ac.uk}}
\runauthor{K. Kaleta and J. L{\H o}rinczi}
\affiliation{Wroc{\l}aw University of Technology and Loughborough University}
\address[A]{Institute of Mathematics\\
\quad and Computer Science\\
Wroc{\l}aw University of Technology\\
Wyb. Wyspia{\'n}skiego 27\\
50-370 Wroc{\l}aw\\
Poland\\
\printead{e1}} 
\address[B]{Department of Mathematical Sciences\\
Loughborough University\\
Loughborough LE11 3TU\\
United Kingdom\\
\printead{e2}}
\end{aug}
\thankstext{T1}{Supported by the National Science Center (Poland) Grant
on the basis of the decision no. DEC-2011/03/D/ST1/00311.}

\received{\smonth{2} \syear{2013}}
\revised{\smonth{9} \syear{2013}}

%
\begin{abstract}
We introduce a class of L\'evy processes subject to specific regularity
conditions, and consider their
Feynman--Kac semigroups given under a Kato-class potential. Using new
techniques, first we analyze the
rate of decay of eigenfunctions at infinity. We prove bounds on
$\lambda
$-subaveraging functions, from
which we derive two-sided sharp pointwise estimates on the ground
state, and obtain upper bounds on all
other eigenfunctions. Next, by using these results, we analyze
intrinsic ultracontractivity and related
properties of the semigroup refining them by the concept of ground
state domination and asymptotic versions.
We establish the relationships of these properties, derive sharp
necessary and sufficient conditions for
their validity in terms of the behavior of the L\'evy density and the
potential at infinity, define the
concept of borderline potential for the asymptotic properties and give
probabilistic and variational
characterizations. These results are amply illustrated by key examples.
\end{abstract}

%
\begin{keyword}[class=AMS]
\kwd[Primary ]{47D08}
\kwd{60G51}
\kwd[; secondary ]{47D03}
\kwd{47G20}
\end{keyword}
\begin{keyword}
\kwd{Symmetric L\'evy process}
\kwd{subordinate Brownian motion}
\kwd{Feynman--Kac semigroup}
\kwd{nonlocal operator}
\kwd{intrinsic ultracontractivity}
\kwd{entropy}
\kwd{ground state domination}
\kwd{decay of eigenfunctions}
\kwd{$\lambda$-subaveraging function}
\end{keyword}

\end{frontmatter}

\section{Introduction}\label{sec1}
Jump L\'evy processes differ in a number of essential ways from
Brownian motion. In this paper,
we focus on two aspects of this qualitatively different behavior under
the effect of a potential
(or penalty function) on the paths. One is a strong smoothing property
of the semigroup of such a
process called intrinsic ultracontractivity. The other is the rate of
decay of its eigenfunctions.
These two properties are related and in this paper we will discuss the
extent of this relationship.

We consider a class of symmetric jump L\'evy processes satisfying
specific regularity conditions.
One condition is given in terms of the convolution of their jump
intensities by a restriction to
a subset of the full state space in relation to large jumps. Another is
existence of transition
densities and their uniform boundedness in space after at least a
sufficiently long time. A final
condition requires sufficient regularity of the Green function for
specially chosen balls. These
conditions are formulated in Assumptions~\ref{assassnu}--\ref{assassbhi} below. As it will be seen,
they are satisfied by important classes of L\'evy processes, including
many cases of interest
of subordinate Brownian motion and also others.

Next, we introduce a potential function $V$ and study the L\'evy
processes perturbed by it in terms of
their Feynman--Kac semigroups $\{T_t\dvtx  t \geq0\}$.
Under a~suitable choice of $V$, which we call $X$-Kato class associated
with L\'evy process $(X_t)_{t \geq0}$
(see Definition~\ref{Xkato} below), the semigroup $\{T_t\dvtx  t \geq0\}$
is well
defined and consists of symmetric
operators. We additionally assume that $V(x) \to\infty$ as $|x| \to
\infty$, which implies that all $T_t$
are compact and have a discrete spectrum. The corresponding
eigenfunctions $\varphi_n$ form an orthonormal
basis in $L^2(\mathbf{R}^d,dx)$ and satisfy $T_t \varphi_n =
e^{-\lambda_n t}
\varphi_n$, where $\lambda_0 <
\lambda_1 \le\lambda_2 \le\cdots\to\infty$. All $\varphi_n$ are
bounded continuous functions, and
each $\lambda_n$ has finite multiplicity. We call $\varphi_0$ ground
state, which can be shown to be
unique and strictly positive.

Since a ground state $\varphi_0$ exists, we can define the intrinsic
Feynman--Kac semigroup
%
\begin{equation}
\label{IFK} \widetilde{T}_t f(x) = \frac{e^{\lambda_0 t}}{\varphi_0(x)}
T_t(\varphi _0 f) (x),
\end{equation}
which\vspace*{1pt} is a Markov semigroup on $L^2(\mathbf{R}^d, \varphi_0^2\,dx)$. Whenever
$T_t$ has an integral kernel $u(t,x,y)$,
the operators $\widetilde{T}_t$ are given by the kernels
%
\begin{equation}
\tilde{u}(t,x,y) = \frac{e^{\lambda_0 t} u(t,x,y)} {\varphi
_0(x)\varphi_0(y)}. \label{kerIFK}
\end{equation}
The intrinsic Feynman--Kac semigroup $\{\widetilde T_t\dvtx  t \geq0\}$
describes a
stationary Markov process which is
called $P(\phi)_1$-process associated with potential $V$ \cite{bibS,bibBL,bibLHB,bibKL}.

Intrinsic ultracontractivity (IUC) means that $\widetilde{T}_t$ is a
bounded operator from $L^2(\mathbf{R}^d,\varphi^2_0\,dx)$
to $L^\infty(\mathbf{R}^d)$ for every $t>0$, or equivalently, $\tilde u(t,x,y)
\le C$ for all $x,y \in\mathbf{R}^d$, with an
appropriate constant $C$ dependent on $V$ and $t$. IUC has been
introduced in \cite{bibDS} for
general semigroups of compact operators and it proved to be a strong
regularity property \cite{bibD}. Important
examples include semigroups of elliptic operators and Schr{\"o}dinger
semigroups on function spaces over $\mathbf{R}^d$ or
bounded domains of $\mathbf{R}^d$ with Dirichlet boundary conditions
\mbox{\cite{bibB,bibD,bibBD,bibKimS1,bibAT,bibMH,bibTom}}.
More recently, IUC has been investigated also in the case of semigroups
generated by fractional Laplacians, and fractional
and relativistic Schr{\"o}dinger operators \mbox{\cite{bibCS1,bibCS2,bibKu2,bibKS,bibKw,bibKaKu,bibKL}}, as well
as for more general symmetric \cite{bibGrz} and nonsymmetric \cite{bibKimS2} L\'evy processes in bounded domains.
In the context of parabolic partial differential equations \cite{bibMur02,bibMur07}, obtained integral representations
of the nonnegative solutions of the Cauchy problem when intrinsic
ultracontractivity holds.

As it follows from our previous analysis, it is of interest to consider
also the property that $\tilde u(t,x,y) \le
C$, for all $x,y \in\mathbf{R}^d$ and sufficiently large $t$ only,
which we
call asymptotic intrinsic ultracontractivity
(AIUC). This is weaker than IUC, and we have seen in the case of
fractional $P(\phi)_1$-processes \cite{bibKL} that
it has an immediate impact on the support properties of their (Gibbs)
path measure. Another important consequence is
that AIUC is equivalent to
\[
\biggl\llvert \frac{e^{\lambda_0 t}u(t,x,y)}{\varphi_0(x)\varphi
_0(y)}-1\biggr\rrvert \leq C e^{-(\lambda_1-\lambda_0) t}, \qquad
t>t_0,
\]
which means that the distribution of the corresponding $P(\phi)_1$-process rapidly tends to equilibrium as $t \to
\infty$ with decay rate given by the spectral gap $\lambda_1-\lambda
_0$. This, in turn, has an offshoot on the
efficiency of practical sampling of conditioned processes; see, for example,
\cite{bibH05,bibH07}. Also, it implies that
the kernel $u(t,x,y)$ takes the shape of the ground state exponentially
quickly, in particular, the decay of the
eigenfunctions $\varphi_1, \varphi_2, \ldots$ will be dominated by the
decay of the ground state.

A basic question we address in this paper is that given the class of
jump L\'evy processes considered, what are
conditions on $V$ making the Feynman--Kac semigroup $\{T_t\dvtx  t \geq0\}$
IUC or
AIUC. The answer is, roughly, that this is
decided by how the asymptotic behaviors of $V$ and $|\log\nu|$ at
infinity compare, where $\nu$ is the L\'evy
density. We further refine IUC-type properties by considering a ground
state domination (GSD) property and its
asymptotic version for sufficiently long times (AGSD). We clarify the
relationships of these properties (Theorem
\ref{thmiucvelgsd}), and give precise necessary and sufficient
conditions (Theorems~\ref{thmsuffiuc} and
\ref{thmneciuc}) for them to hold. Our results recover the facts on
IUC previously known for stable processes
\cite{bibKaKu,bibKL} and relativistic stable processes \cite{bibKS}
only, and establish these properties for
many other processes, also shedding new light on existing results for
diffusions \cite{bibDS,bibD}. Corollary
\ref{corborder} gives a sharp description of the
borderline potential $V(x) \asymp|\log\nu(x)|$ distinguishing (A)IUC
from non(A)IUC. In comparison with the classic
result which says roughly that IUC holds for a diffusion when $V$ grows
in leading order super-quadratically, it is
seen that for a jump L\'evy process it is ``easier'' to be (A)IUC than
for a diffusion. We give an explanation of this
in terms of a balance mechanism between the competition of killing and
survival of paths (Proposition~\ref{propprobchar}),
and give a probabilistic characterization (Propositions~\ref{propAIUCproba}~and~\ref{propIUCproba}). Furthermore (Theorem
\ref{thmcharacter}, Corollary~\ref{variational}), we obtain a second
characterization of AIUC in terms of minimizing
a free energy functional appearing as the difference of an energy and
an entropy associated with the L\'evy measure
[see (\ref{energ})--(\ref{free}) below], and obtain the borderline
potential as the solution of this variational problem.
Due to the role played by the entropy this also explains why $\log\nu$
appears in this expression.

A second basic problem we address is to derive pointwise bounds on the
eigenfunctions for a given L\'evy process
and $V$. We obtain sharp lower and upper bounds (Theorem~\ref{thmgsest}) showing that the fall-off of the ground
state follows the tail behavior of the L\'evy density with corrections
resulting from the contribution of the
potential. Furthermore, we obtain upper bounds on all other
eigenfunctions (Theorem~\ref{thmbsest}, Corollary
\ref{surp1}). Under a reasonable condition, we derive a more explicit
expression of the dependence of the decay
on $V$ (Corollary~\ref{comparable}). We note that, importantly, the
ground state estimates follow without any
need to use results on the (A)IUC properties, unlike in the previous
work \cite{bibKS}. Our considerations
lead naturally to studying $\lambda$-subaveraging functions, which can
be thought of as counterparts to
$\lambda$-superaveraging functions known in potential theory, and we
prove two results on them (Theorems
\ref{thmdefic1} and~\ref{thmdefic2}).
Although it makes the paper longer, we find it useful to discuss
important (classes of) examples in relation
to both the ground state bounds and the IUC-type properties. We also
discuss in some detail which types of
specific cases are covered by the L\'evy processes we tackle in this
paper, as well as interesting cases which
fail the assumptions or the IUC-properties.

We note that our results can also be considered from the perspective of
the correspondence between jump L\'evy
processes and nonlocal operators, which are their generators. Via a
Feynman--Kac representation our results
characterize the decay of eigenfunctions and IUC-type properties of
semigroups related to generalized Schr\"odinger operators whose kinetic
terms are given by the generators. A
specific class of nonlocal operators covered to a large extent by our
results are of the form $\Psi(-\Delta)
+ V$, studied in \cite{bibHIL,bibHL}, where $\Psi$ is a Bernstein
function with vanishing right limit at
zero, giving the L\'evy exponent of a\vspace*{1pt} subordinator \cite{bibBF,bibSSV}. Some important specific cases are
fractional Schr\"odinger operators $(-\Delta)^{\alpha/2} + V$,
relativistic Schr\"odinger operators $(-\Delta+
m^{2/\alpha})^{\alpha/2} - m + V$, jump-diffusion operators
$a(-\Delta
)^{\alpha/2} - b\Delta+ V$, and many others.
There is an increasing literature studying these operators from both a
probabilistic and an analytic point of view
\cite{bibCMS,bibBB1,bibBB2,bibSong,bibCS,bibKS,bibFLS,bibLS,bibKaKu,bibKL,bibLMa,bibKal,bibCKSo,bibBYY,bibKu3}.

The basic input object in this paper is a L\'evy process, therefore, we
mainly use probabilistic methods. Our argument
builds on a completely new approach which combines sharp uniform
estimates on the local extrema of functions harmonic with
respect to the subprocess obtained under the Feynman--Kac functional of
the L\'evy process (Lemma~\ref{lmbhiold})
developed only recently in \cite{bibBKK} (see also \cite{bibKSV1}),
and new powerful self-improving estimates
under the path measure of the process (see the proofs of Theorems~\ref{thmdefic1} and~\ref{thmcruiuc}).
In the proofs, it will become evident that the pivotal Assumption~\ref{assassnu}(3) is very natural, and its
generality will be seen by many examples of interest satisfying it. In
particular, this will allow to study
also processes with exponentially localized L\'evy measures and derive
sharp estimates, which are of special
interest for various further investigations and have been little
understood before (see \cite{bibCMS}).
Since IUC has been much studied in operator semigroup and PDE context,
we find it important to develop
a probabilistic understanding of it.

The remainder of the paper is organized as follows. In Section~\ref{sec2}, we
state the main results. First, we introduce
the class of L\'evy processes and potentials which will be considered.
The next two subsections are devoted to
presenting the estimates on $\lambda$-subaveraging functions, ground
states and the other eigenfunctions. The
last two subsections present the results on intrinsic
ultracontractivity and ground state domination. In Section~\ref{sec3}, we
provide the proofs in a similar division of the material. We
devote Section~\ref{sec4} to a detailed discussion of
ground state decay and IUC-type behaviors of specific processes of interest.

\section{Assumptions and main results}\label{sec2}

\subsection{A class of L\'evy processes}\label{sec2.1}

Let $(X_t)_{t \geq0}$ be a L\'evy process in $\mathbf{R}^d$, $d \geq
1$. We use the
notations $\mathbf{P}^x$ and $\mathbf{E}^x$,
respectively, for the probability measure and expected value of the
process starting in $x \in\mathbf{R}^d$.
Recall that $(X_t)_{t \geq0}$ is a Markov process with respect to its own
filtration satisfying the strong Markov
property, and has right continuous paths with left limits (c\`adl\`ag
paths). It is a basic fact that
$(X_t)_{t \geq0}$ is completely determined by its characteristic
exponent $\psi$
given by the L\'evy--Khintchine
formula, that is, for $\xi\in\mathbf{R}^d$
\[
\mathbf{E}^0 \bigl[e^{i \xi\cdot X_t} \bigr] = e^{-t \psi(\xi)}
\]
holds with
\[
\psi(\xi) = - i \gamma\cdot\xi+ A \xi\cdot\xi+ \int_{\mathbf
{R}^d}
\bigl(1-e^{i \xi\cdot z} + i \xi\cdot z \mathbf{1}_{ \{|z|<1 \}}(z) \bigr)
\nu(dz).
\]
Here, $\gamma\in\mathbf{R}^d$ (\emph{drift coefficient}), $A$ is a symmetric
nonnegative definite
$d \times d$ matrix (\emph{diffusion} or \emph{Gaussian coefficient}),
and $\nu$ is a Radon
measure on $\mathbf{R}^d \setminus \{0 \}$ such that
$\int_{\mathbf{R}^d} (1
\wedge|z|^2) \nu(dz)
< \infty$ (\emph{L\'evy measure}). The defining parameters $(\gamma, A,
\nu)$ are called the
\emph{L\'evy triplet} of the process. $(X_t)_{t \geq0}$ is said to be
\emph
{symmetric} whenever $X_t$
has the same distribution as $-X_t$ for all $t>0$. In this case, $\psi
(\xi) = \psi(-\xi)$, $\xi
\in\mathbf{R}^d$, that is, $\gamma\equiv0$ and $\nu(B)= \nu(-B)$
for all $B
\in\mathcal{B}(\mathbf{R}^d \setminus \{0 \})$, and
then the characteristic exponent reduces to
\[
\psi(\xi) = A \xi\cdot\xi+ \int_{\mathbf{R}^d} \bigl(1-\cos(\xi \cdot
z)\bigr) \nu(dz).
\]
Whenever $\nu(dz)$ is absolutely continuous with respect to Lebesgue
measure, we denote its density
by $\nu(z)$ and call it the \emph{L\'evy} (\emph{jump}) \emph{intensity} of
$(X_t)_{t \geq0}$.
When $A \equiv0$ and $\nu
\neq0$, the process $(X_t)_{t \geq0}$ is said to be a \emph{purely
jump} L\'evy
process. For more details
on L\'evy processes, we refer to \cite{bibBer,bibApp,bibSat,bibFri,bibJ}.

We will use throughout the notation $C(a,b,c,\ldots)$ for a positive
constant dependent on
parameters $a,b,c,\ldots,$ while dependence on the process $(X_t)_{t
\geq0}$ is
indicated by $C(X)$, and
dependence on the dimension $d$ is assumed without being indicated.
Since constants appearing
in definitions, lemmas and theorems play a role later on, we use the
numbering $C_1, C_2, \ldots$
to be able to track them. We will also use the notation $f \asymp C g$
meaning that $C^{-1} g
\leq f \leq Cg$ with a constant $C$, while $f \asymp g$ means that
there is a constant $C$ such
that the latter holds.

For the remainder of the paper, we make three standing assumptions on
the L\'evy processes we
consider.

\begin{assumption}
\label{assassnu}
$(X_t)_{t \geq0}$ is a symmetric L\'evy process with L\'evy measure absolutely
continuous with respect to Lebesgue
measure. The corresponding jump intensity $\nu(x)$ is strictly positive
and satisfies the following three
conditions:
\begin{longlist}[(2)]
\item[(1)] For every $0<r \leq1/2$ there is a constant $C_1=C_1(X,r)
\geq1$ such that
\[
\nu(x) \asymp C_1 \nu(y), \qquad r \leq|y| \leq|x| \leq|y|+ 1.
\]
\item[(2)] There is a constant $C_2=C_2(X) \geq1$ such that
\[
\nu(x) \leq C_2 \nu(y), \qquad1/2 \leq|y| \leq|x|.
\]
\item[(3)] There is a constant $C_3=C_3(X) \geq1$ such that
\[
\int_{|z-x|>1/2,  |z-y|>1/2} \nu(x-z) \nu(z-y) \,dz \leq C_3 \nu
(x-y), \qquad|y-x|\geq1.
\]
\end{longlist}\vspace*{6pt}
Conditions (1) and (2) are clearly satisfied when, for example, $\nu(x)
\asymp\kappa(|x|)$, $x \in\mathbf{R}^d$, where
$\kappa\dvtx (0,\infty) \to(0,\infty)$ is a nonincreasing function such
that $\kappa(s) \leq C \kappa(s+1)$,
$s \geq1/2$. Condition (3) provides a regularity of the convolutions
of $\nu$ with respect to large jumps.
While (1)--(2) can be seen as technical conditions, (3)~has a
structural importance. Examples and counterexamples to
conditions \mbox{(1)--(3)} above are discussed in Section~\ref{sec4.1}.

Denote by $P_t f(x) = \mathbf{E}^x [f(X_t)]$ the transition operators
of the
process $(X_t)_{t \geq0}$. Recall that
$(X_t)_{t \geq0}$ has the strong Feller property if $P_t(L^{\infty
}(\mathbf{R}^d))
\subset C_{\mathrm{b}}(\mathbf{R}^d)$, for all $t>0$.
\end{assumption}

\begin{assumption}
\label{assassdensity}
The process $(X_t)_{t \geq0}$ has the strong Feller property, or equivalently,
the one-dimensional $\mathbf{P}^x$-distributions
of $(X_t)_{t \geq0}$ are absolutely continuous with respect to
Lebesgue measure,
that is, there exist transition probability
densities $p(t,x,y)=p(t,y-x,0) =: p(t,y-x)$. Furthermore, there exist
$t_{b}>0$ and $C_4=C_4(X,t_{b})$ such that
$0<p(t_{b},x) \leq C_4$, for all $x \in\mathbf{R}^d$.
\end{assumption}

Note that the first part of the above assumption is satisfied by a
large class of L\'evy processes including
subordinate Brownian motion \cite{bibKSV} provided that they are not
compound Poisson processes. In fact,
our assumption is equivalent to $e^{-t_{b} \psi(\cdot)} \in
L^1(\mathbf{R}^d)$,
for some $t_{b}>0$. In this case,
$p(t_{b},x)$ can be obtained by the Fourier inversion formula. Clearly,
this property extends to all $t \geq
t_{b}$ by the Markov property of $(X_t)_{t \geq0}$. For more details
on the
existence and properties of transition
probability densities for L\'evy processes, we refer to \cite{bibKSch}
and references therein.

We note for later use that under Assumption~\ref{assassdensity}
transition densities $p_D(t,x,y)$ of the
process $(X_t)_{t \geq0}$ killed upon exiting an open bounded set $D
\subset\mathbf{R}
^d$ also exist. In this case, the Hunt
formula
%
\begin{eqnarray}\label{eqHuntF}
p_D(t,x,y)=p(t,y-x) - \mathbf{E}^x \bigl[
\tau_D < t; p(t-\tau_D, y-X_{\tau
_D}) \bigr],
\nonumber\\[-10pt]\\[-10pt]
\eqntext{x, y \in D,}
\end{eqnarray}
holds, where $\tau_D = \inf \{t \geq0\dvtx  X_t \notin D \}$ is
the first exit time from $D$. The Green
function of the process $(X_t)_{t \geq0}$ on $D$ is thus given by
$G_D(x,y)= \int_0^{\infty} p_D(t,x,y) \,dy$, for all
$x, y \in D$, and $G_D(x,y)=0$ if $x \notin D$ or $y \notin D$.

Since our results rely on a use of potential theory, we need some more
regularity of $(X_t)_{t \geq0}$.

\begin{assumption}
\label{assassbhi}
For all $0<p<q<R \leq1$, we have
\[
\sup_{x \in B(0,p)} \sup_{y \in
B(0,q)^c} G_{B(0,R)}(x,y) < \infty.
\]
\end{assumption}

In many cases, Assumption~\ref{assassbhi} is a direct consequence of
time--space estimates of the function
$p(t,x)$. Indeed, it is clearly satisfied when the boundedness
condition holds with $G_{B(0,R)}(x,y)$
replaced by the potential kernel $G_{\mathbf{R}^d}(x, y)= \int_0^{\infty
}p(t,x,y)\,dt$ or, as proved in
\cite{bibBKK}, Proposition 2.3, the $\lambda$-potential\vspace*{1pt} kernel
$G^{\lambda}_{\mathbf{R}^d}(x, y)=
\int_0^{\infty}e^{-\lambda t} p(t,x,y)\,dt$, $\lambda>0$, whenever the
process $(X_t)_{t \geq0}$ is recurrent.

One of our key arguments following below uses some estimates (see Lem\-ma~\ref{lmbhiold}) on the local
extrema of functions harmonic with respect to the subprocess of
$(X_t)_{t \geq0}$ obtained under its Feynman--Kac
functional. These bounds are a direct consequence of more general
results of Bogdan, Kumagai and
Kwa\'snicki obtained recently in \cite{bibBKK}. To borrow these
results, we need to match some
assumptions made in this paper, however, since we consider symmetric L\'
evy processes, Assumptions~\ref{assassnu}(1),~\ref{assassdensity} [without requiring
boundedness of $p(t,x)$] and~\ref{assassbhi}
provide sufficient regularity of $(X_t)_{t \geq0}$ to allow a use of
\cite{bibBKK}. The remaining conditions in
Assumption~\ref{assassnu} are independent from this context and
together with condition (1) for
$r=1/2$ only they will allow to draw more regularity of the L\'evy
intensity $\nu$ needed in controlling
jumps in Section~\ref{subsecjumpest} below. Similarly, boundedness
of $p(t,x)$ in Assumption
\ref{assassdensity} guarantees sufficient regularity of the process
needed below.

Note that all of our assumptions above are satisfied by a wide class of
symmetric L\'evy processes including a large subclass of subordinate
Brownian motions, L\'evy processes with nondegenerate Brownian
components, symmetric stable-like ones or processes with
subexponentially localized L\'evy measures. Some important examples
with a verification of assumptions are discussed in Section~\ref{subsecass}.

Next we give the class of potentials which will be used in this paper.

\begin{definition}[($X$-Kato class)]\label{Xkato}
We say that the Borel function $V\dvtx\break  \mathbf{R}^d \to\mathbf{R}$
belongs to the
Kato-class $\mathcal{K}^X$ associated with the
L\'evy process $(X_t)_{t \geq0}$ if it satisfies
%
\begin{equation}
\label{eqKatoclass} \lim_{t \rightarrow0} \sup_{x \in\mathbf{R}^d}
\mathbf{E}^x \biggl[\int_0^t
\bigl|V(X_s)\bigr| \,ds \biggr] = 0.
\end{equation}
We write $V \in\mathcal{K}_{\operatorname{loc}}^X$ if $V \mathbf
{1}_B \in\mathcal{K}^X$ for every ball $B
\subset\mathbf{R}^d$. Moreover, we say
that $V$ is an \emph{$X$-Kato decomposable potential}, whenever
\[
V = V_{+} - V_{-}\qquad\mbox{with } V_{-} \in
\mathcal{K}^X, V_{+} \in\mathcal{K}_{\operatorname{loc}}^X,
\]
where $V_{+}$ and $V_{-}$ denote the positive and negative parts of
$V$, respectively.
\end{definition}

For simplicity, in what follows we refer to $X$-Kato decomposable
potentials as \emph{$X$-Kato class
potentials}. It is easy to see that $L^{\infty}_{\operatorname{loc}}
\subset\mathcal{K}_{\operatorname{loc}
}^X$. Moreover, by stochastic
continuity of $(X_t)_{t \geq0}$ also $\mathcal{K}_{\operatorname
{loc}}^X \subset L^1_{\operatorname{loc}}(\mathbf{R}^d)$,
and thus an $X$-Kato class potential
is always locally absolutely integrable. Note that condition (\ref
{eqKatoclass}) allows local singularities
of $V$. For specific processes $(X_t)_{t \geq0}$ condition (\ref
{eqKatoclass})
can be equivalently reformulated in
terms of the potential kernel of the process in the transient case, and
the so-called compensated potential
kernel when the process is recurrent (for more details see, e.g., \cite{bibZ,bibCMS,bibBB2}).

We single out a restricted set of potentials which will be often used below.

\begin{assumption}
\label{assasspinning}
For a given L\'evy process $(X_t)_{t \geq0}$ let $V$ be such that:
\begin{longlist}[(2)]
\item[(1)] $V$ is an $X$-Kato class potential

\item[(2)] $V(x) \to\infty$ as $|x| \to\infty$.
\end{longlist}
\end{assumption}


Next, for a given $X$-Kato class potential $V$, we define
\[
T_t f(x) = \mathbf{E}^x \bigl[e^{-\int_0^t V(X_s) \,ds}
f(X_t) \bigr], \qquad f \in L^2\bigl(
\mathbf{R}^d\bigr), t>0.
\]
Using the Markov property and stochastic continuity of the process, it
can be shown that
$\{T_t\dvtx  t\geq0\}$ is a strongly continuous semigroup of symmetric
operators on $L^2(\mathbf{R}^d)$, which
we call the \emph{Feynman--Kac semigroup} associated with the process
$(X_t)_{t \geq0}$ and potential $V$. In
particular, by the Hille--Yoshida theorem, there exists a self-adjoint
operator $H$ bounded from below
such that $e^{-tH} = T_t$. The operator $H$ is often seen as a \emph
{generalized Schr\"odinger operator}
based on the infinitesimal generator $L$ of the process $(X_t)_{t \geq0}$.
Whenever $V$ is relatively bounded
with respect to $L$ with relative bound less than 1 we can write $H =
-L+V$ as an operator sum.

We now summarize the basic properties of the operators $T_t$, some of
which will be explicitly used below.

\begin{lemma}
\label{lmsemprop}
Let Assumptions~\ref{assassnu}--\ref{assasspinning} be satisfied.
Then the following properties hold:
\begin{longlist}[(2)]
\item[(1)]
For all $t>0$, $T_t$ are bounded operators on each $L^p(\mathbf
{R}^d)$, $1 \leq
p \leq\infty$. The operators $T_t\dvtx
L^p(\mathbf{R}^d) \to L^p(\mathbf{R}^d)$ for $1 \leq p \leq\infty$,
$t > 0$, $T_t\dvtx
L^p(\mathbf{R}^d) \to L^{\infty}(\mathbf{R}^d)$
for $1 < p \leq\infty$, $t \geq t_{b}$, and $T_t\dvtx  L^1(\mathbf{R}^d)
\to
L^{\infty}(\mathbf{R}^d)$ for $t \geq2t_{b}$ are
bounded.

\item[(2)]
For all $t \geq2t_{b}$, $T_t$ has a bounded, measurable, and symmetric
kernel $u(t, \cdot, \cdot)$, that is,
$T_t f(x) = \int_{\mathbf{R}^d} u(t,x,y)f(y) \,dy$, $f \in L^p(\mathbf
{R}^d)$, $1 \leq p
\leq\infty$.

\item[(3)]
For all $t>0$ and $f \in L^{\infty}(\mathbf{R}^d)$, $T_t f(x)$ is a bounded
continuous function.

\item[(4)]
For all $t \geq2t_{b}$ the operators $T_t$ are positivity improving,
that is,\break $T_t f(x) > 0$ for all
$x \in\mathbf{R}^d$ and $f \in L^2(\mathbf{R}^d)$ such that $f \geq
0$ and $f \neq0$ a.e.

\item[(5)]
All operators $T_t\dvtx L^2(\mathbf{R}^d) \to L^2(\mathbf{R}^d)$, $t>0$,
are compact.
\end{longlist}
\end{lemma}

Properties (1)--(4) can be established by standard arguments based on
\cite{bibCZ}, Section~3.2. Property
(5) is a consequence of $V(x) \to\infty$ as $|x| \to\infty$ and
standard arguments based on approximation
of $T_{t}$, $t \geq2t_{b}$, by compact operators,\vspace*{1pt} see \cite{bibKL}, Lemma~3.2. Clearly, compactness extends to
all $t>0$ by the fact that $T_t = e^{-tH}$ for a self-adjoint operator
$H$ and a use of the spectral theorem.
Note that we do not assume that $p(t,x)$ is bounded for all $t>0$, and
thus in general the operators
$T_t\dvtx L^p(\mathbf{R}^d) \to L^{\infty}(\mathbf{R}^d)$ need not be
bounded for $t<t_{b}$.

The theory of operator semigroups implies that there exists an
orthonormal basis in $L^2(\mathbf{R}^d)$ consisting
of the eigenfunctions $\varphi_n$ given by $T_t \varphi_n =
e^{-\lambda
_n t} \varphi_n$, $t>0$, $n \geq0$,
and $\lambda_0 < \lambda_1 \leq\lambda_2 \leq\cdots\to\infty$. All
$\varphi_n$ are bounded continuous
functions. Moreover, the first eigenfunction (or ground state) $\varphi
_0$ has a strictly positive version (\cite{bibRS}, Theorem XIII.43), which will be our choice throughout.

\subsection{Estimates of \texorpdfstring{$\lambda$}{$lambda$}-subaveraging functions}\label{sec2.2}

Recall that one of the fundamental objects in potential theory are
\emph{$\lambda$-superaveraging} (and
related \emph{$\lambda$-excessive}) functions, see  \cite{bibCW}, Section~2.1. Below it is useful to consider
\emph{$\lambda$-subaveraging} functions, which in some sense are
counterparts of
$\lambda$-superave\-raging functions in the opposite direction of
domination. We say that a nonnegative
Borel function $\varphi$ is \emph{$\lambda$-subaveraging} for the
semigroup $\{T_t\dvtx  t \geq0\}$ with $\lambda\geq0$
if for every $t>0$ and $x \in\mathbf{R}^d$ we have $e^{\lambda t} T_t
\varphi
(x) \geq\varphi(x)$.

For an open set $D \subset\mathbf{R}^d$, a Kato-class potential $V$
and a
nonnegative or bounded Borel function
$\varphi$ we define the \emph{$V$-Green operator} for the semigroup
$\{T_t\dvtx  t \geq0\}$ and set $D$ (see Section~\ref{subsecprel}),
\[
G^V_D \varphi(x) = \mathbf{E}^x \biggl[
\int_0^{\tau_D} e^{-\int_0^t
V(X_s)\,ds}\varphi(X_t)\,dt
\biggr], \qquad x \in D,
\]
where $\tau_D$ is the first exit time from $D$.

The following estimates for $\lambda$-subaveraging functions will be
used in proving bounds on eigenfunctions and intrinsic ultracontractivity.

\begin{theorem}
\label{thmdefic1}
Let Assumptions~\ref{assassnu}--\ref{assasspinning} hold. If
$\varphi
$ is a bounded $\lambda$-subavera\-ging
function for the semigroup $\{T_t\dvtx  t \geq0\}$ with $\lambda\geq0$,
then there
is a constant $C_4=C_4(X,V,\lambda)$
and $R=R(X,V,\lambda)>0$
such that
\[
\varphi(x) \leq C_4 \llVert \varphi\rrVert _{\infty} \nu(x),
\qquad|x| \geq R.
\]
\end{theorem}

The proof of this theorem is probably the most involved and crucial
part of the paper. The required bound is
obtained inductively, stemming from a new idea based on a
self-improving estimate iterated infinitely many
times. The main difficulty is that we need to have a statement on
strictly $\nu(x)$ rather than $\nu(cx)$
for some $c \in(0,1)$. This is particularly critical in the case of
exponentially localized L\'evy measures,
which are of special interest in our further investigations.

For simplicity, we write $\mathbf{1}(x)$ instead of $\mathbf
{1}_{\mathbf{R}^d}(x)$ throughout below.

\begin{theorem}
\label{thmdefic2}
Let Assumptions~\ref{assassnu}--\ref{assasspinning} hold.
\begin{longlist}[(2)]
\item[(1)]
If $\varphi$ is a bounded function (possibly negative) for which there
exists $\lambda>0$ such that
for every $t>0$ and $x \in\mathbf{R}^d$ we have $e^{\lambda t} T_t
\varphi
(x)=\varphi(x)$ (clearly, in this
case $|\varphi|$ is $\lambda$-subaveraging), then there is a constant
$C_5=C_5(X,V,\lambda)$ and $R=
R(X,V,\lambda)>0$ such that
\[
\bigl|\varphi(x)\bigr| \leq C_5 \llVert \varphi\rrVert _{\infty}
G^V_{B(x,1)}\mathbf{1}(x) \nu(x), \qquad|x| \geq R.
\]
\item[(2)]
If $\varphi$ is a strictly positive function for which there is
$\lambda> 0$ such that for every $t>0$
and $x \in\mathbf{R}^d$ we have $e^{\lambda t} T_t\varphi(x) =
\varphi(x)$,
then there is a constant $C_6=C_6(X,\varphi)$
and $R=R(X,V,\lambda)>0$ such that
\[
\varphi(x) \geq C_6 G^V_{B(x,1)}\mathbf{1}(x)
\nu(x), \qquad|x| \geq R.
\]
\end{longlist}
\end{theorem}

\subsection{Eigenfunction estimates}\label{sec2.3}

The following pointwise upper bounds for eigenfunctions and sharp
two-sided bounds for the ground state of
the operators $T_t$ are the next main results of this paper.

\begin{theorem}[(Upper bounds on eigenfunctions)]\label{thmbsest}
If Assumptions~\ref{assassnu}--\ref{assasspinning} hold, then for
every $n \in \{0, 1, 2, \ldots \}$ and $\eta\geq0$ such that
$\lambda_0 + \eta>0$, there exists a constant $C_7=C_7(X,V,n,\eta)$
and a radius $R=R(X,V,n,\eta)>0$ such that
\[
\bigl|\varphi_n(x)\bigr| \leq C_7 G^{V+\eta}_{B(x,1)}
\mathbf{1}(x) \nu(x), \qquad|x| \geq R.
\]
\end{theorem}

\begin{theorem}[(Ground state estimates)]\label{thmgsest}
If Assumptions~\ref{assassnu}--\ref{assasspinning} hold, then for
every $\eta\geq0$ such that
$\lambda_0 + \eta>0$ there exist constants $C_8=C_8(X,V,\eta)$,
$C_9=C_9(X,V,\eta)$ and a radius $R=R(X,V,\eta)>0$
such that
\[
C_8 G^{V+\eta}_{B(x,1)}\mathbf{1}(x) \nu(x) \leq
\varphi _0(x) \leq C_9 G^{V+\eta}_{B(x,1)}
\mathbf{1}(x) \nu(x), \qquad|x| \geq R.
\]
\end{theorem}

We emphasize that the above bounds on the eigenfunctions are obtained
by using a completely new idea in this context, without
using any (intrinsic) ultracontractivity properties of $\{T_t\dvtx  t \geq
0\}$,
unlike in \cite{bibKS}, which will be further discussed below.

The following domination property is an immediate consequence of the
above theorems. We note that this is in contrast with Brownian motion,
for which it does not occur if the growth of the potential $V$ at
infinity is not fast enough [see further discussion in Example~\ref{exex6}(5) and compare with (\ref{eqiucdomination}) below].

\begin{corollary}\label{surp1}
\label{cordom}
If Assumptions~\ref{assassnu}--\ref{assasspinning} hold, then for
every $n \in \{1, 2, \ldots \}$ there is a constant
$C_{10}=C_{10}(X,V,n)$
such that
\[
\bigl|\varphi_n(x)\bigr| \leq C_{10} \varphi_0(x),
\qquad x \in\mathbf{R}^d.
\]
\end{corollary}

By the estimates in (\ref{eqgsVest}), we also have the following corollary.

\begin{corollary}
\label{comparable}
Let Assumptions~\ref{assassnu}--\ref{assasspinning} hold. Then for
every $n \in \{0,1, 2, \ldots \}$ there exists a radius
$R=R(X,V,n)>0$ such that
\[
\bigl|\varphi_n(x)\bigr| \leq C_7 \frac{\nu(x)}{\inf_{y \in B(x,1)} V(y)}, \qquad|x|
\geq R
\]
and
\[
C_{11} \frac{\nu(x)}{\sup_{y \in B(x,1)} V(y)} \leq \varphi _0(x) \leq
C_9 \frac{\nu(x)}{\inf_{y \in B(x,1)} V(y)}, \qquad|x| \geq R
\]
with some constant $C_{11}=C_{11}(X,V)$. In particular, if for some $n
\in \{0,1, 2, \ldots \}$ there is a constant $C>1$ such
that for
all unit balls $B \subset B(0,R)^c$ it holds
that $\sup_{y \in B} V(y) \leq C \inf_{y \in B} V(y)$ (cf. Assumption
\ref{assasscomp} below), then
\[
\bigl|\varphi_n(x)\bigr| \leq C_7 C \frac{\nu(x)}{V(x)},
\qquad|x| \geq R+1
\]
and
\[
C_{11} C^{-1} \frac{\nu(x)}{V(x)} \leq \varphi_0(x)
\leq C_9 C \frac{\nu(x)}{V(x)}, \qquad|x| \geq R+1.
\]
\end{corollary}

Ground state decays for specific examples are discussed in Section~\ref{sec4.2} below.

\subsection{Intrinsic ultracontractivity and ground state domination}\label{sec2.4}
\label{subsecIUC}
Under the given choice of potential the Feynman--Kac semigroup has
strong smoothing properties which we define
next. In particular, they imply degrees of regularity and the rate of
decay of eigenfunctions. Recall that the
intrinsic Feynman--Kac semigroup is given by (\ref{IFK}).

\begin{definition}[(IUC/AIUC)]
\label{defIUC}
Consider the following ultracontractivity properties:
\begin{longlist}[(2)]
\item[(1)]
The semigroup $\{T_t\dvtx  t \geq0\} $ is \emph{ultracontractive} if $T_t$
is a
bounded operator from $L^2(\mathbf{R}^d)$ to
$L^{\infty}(\mathbf{R}^d)$, for every $t>0$.

\item[(2)]
The semigroup $\{T_t\dvtx  t \geq0\} $ is \emph{intrinsically ultracontractive}
(abbreviated as~IUC) if $\widetilde T_t$ is a
bounded operator from $L^2(\mathbf{R}^d,\varphi_0^2\,dx)$ to
$L^{\infty}(\mathbf{R}^d)$,
for every $t>0$.

\item[(3)]
The\vspace*{1pt} semigroup $\{T_t\dvtx  t \geq0\}$ is \emph{$t_0$-intrinsically
ultracontractive}
(abbreviated as $t_0$-IUC) if the above
boundedness property of $\widetilde T_t$ holds for some $t_0 >0$. In
this case, the semigroup property extends
it to all $t \geq t_0$.

\item[(4)]
When $\{T_t\dvtx  t \geq0\}$ is $t_0$-IUC, but the specific value of $t_0$
is not
essential, we simply say that $\{T_t\dvtx  t \geq0\}$
is \emph{asymptotically intrinsically ultracontractive} (abbreviated
as AIUC).
\end{longlist}
\end{definition}

A remarkable consequence of IUC-type properties is the following
domination property for eigenfunctions. If for some $t>0$, the
semigroup $\{T_t\dvtx  t \geq0\}$ is $t$-IUC, then there is a constant $C=C(X,V,t)$
such that (see, e.g., \cite{bibB}, (1.7))
%
\begin{equation}
\label{eqiucdomination} \bigl|\varphi_n(x)\bigr| \leq C e^{(\lambda_n-\lambda_0) t}
\varphi_0(x), \qquad x \in\mathbf{R}^d,  n \geq1.
\end{equation}
Clearly, if $\{T_t\dvtx  t \geq0\}$ is IUC, then (\ref{eqiucdomination})
holds for
all $t>0$. Unlike in Corollary~\ref{cordom}, here the dependence on
$n$ of the expression on the right-hand side of the inequality is more explicit.

Since below we mainly use probabilistic arguments, it will be useful to
consider the following property of the
semigroup $\{T_t\dvtx  t \geq0\}$ which we call \emph{ground state
domination}. As it
will be seen later, in general ground
state domination is a weaker property than IUC.

\begin{definition}[(GSD/AGSD)]
Consider the following boundedness properties:
\begin{longlist}[(2)]
\item[(1)]
The semigroup $\{T_t\dvtx  t \geq0\}$ is \emph{ground state dominated} (abbreviated
as GSD) if for every $t>0$ there is a
constant $C_{12}=C_{12}(X,V,t)$ such that
%
\begin{equation}
\label{eqdefiuc} T_t \mathbf{1}(x) \leq C_{12}
\varphi_0(x), \qquad x \in\mathbf{R}^d.
\end{equation}
\item[(2)]
The semigroup $\{T_t\dvtx  t \geq0\}$ is \emph{$t_0$-ground state dominated}
(abbreviated as \mbox{$t_0$-}GSD) if (\ref{eqdefiuc}) holds
for some $t_0 >0$. In this case, the semigroup property extends this
bound to all $t
\geq t_0$
with constant $C_{12} e^{-\lambda_0 (t-t_0)}$, where $C_{12}=C_{12}(X,V,t_0)$.
\item[(3)]
When $\{T_t\dvtx  t \geq0\}$ is $t_0$-GSD but the specific value of $t_0$
is not
essential, we simply say that $\{T_t\dvtx  t \geq0\}$
is \emph{asymptotically ground state dominated} (abbreviated as AGSD).
\end{longlist}
\end{definition}

Before stating the main results of this subsection, it is worthwhile to
discuss the relationship between these
properties. Parts (1)--(2) in Definition~\ref{defIUC} are standard, (4)
has been introduced in \cite{bibKL}. It is
immediate from the definitions that IUC and GSD imply $t_0$-IUC and
$t_0$-GSD (for all $t_0>0$), respectively. However,
as it will be seen below, IUC and GSD are essentially stronger
properties than their asymptotic versions. We will show
that under our assumptions AIUC and AGSD are equivalent, while IUC
implies GSD and in general conversely this is not
the case.

\begin{theorem}[{[(A)IUC and (A)GSD]}]
\label{thmiucvelgsd}
Let Assumptions~\ref{assassnu}--\ref{assasspinning} be satisfied,
specifically, let Assumption~\ref{assassdensity}
hold with $t_{b}>0$.
\begin{longlist}[(2)]
\item[(1)]
Then
\[
\mathrm{AIUC} \Longleftrightarrow \mathrm{AGSD}\quad\mbox {and}\quad
\mathrm{IUC} \Longrightarrow \mathrm{GSD}
\]
in the sense that $t_0$-IUC $\Longrightarrow$ $2t_0$-GSD, $t_0 > 0$,
and $t_0$-GSD $\Longrightarrow$ $2t_0$-IUC,
whenever $t_0 \geq t_{b}$.

\item[(2)]
If, moreover, $p(t,x) \leq C_4$ for all $t>0$ and $x \in\mathbf{R}^d$ with
$C_4=C_4(X,t)$, then also
\[
\mathrm{GSD} \Longrightarrow \mathrm{IUC}.
\]
\end{longlist}
\end{theorem}

In Proposition~\ref{propultracontractivity} below, we show that the
assumption in part (2) of Theorem~\ref{thmiucvelgsd} is essential. This means that in general when
$p(t,\cdot)$ may be unbounded for
small $t$, IUC is a strictly stronger property than GSD. Intuitively,
it is clear that IUC requires
more smoothness of the semigroup $\{T_t\dvtx  t \geq0\}$ than GSD as it
also depends
on the local singularities of
the semigroup, while GSD is in fact, roughly speaking, a decay property
of the semigroup at infinity.

We now present characterization results on GSD/AGSD and IUC/AIUC.

\begin{theorem}[(Sufficient and necessary conditions for AGSD)]
\label{thmsuffiuc}
Let Assumptions~\ref{assassnu}--\ref{assasspinning} hold.
\begin{longlist}[(2)]
\item[(1)]
If there exist a constant $C_{13}=C_{13}(X,V)$ and a radius $R > 0$
such that
\[
\frac{V(x)}{|\log\nu(x)|} \geq C_{13}, \qquad|x| \geq R,
\]
then the semigroup $\{T_t\dvtx  t \geq0\}$ is $t_0$-GSD with $t_0 = 4/C_{13}$.

\item[(2)]
If the semigroup $\{T_t\dvtx  t \geq0\}$ is $t_0$-GSD, then for every
$\varepsilon
\in(0,1]$ there is $R_{\varepsilon} > 0$
such that
\[
\frac{\sup_{y \in B(x,\varepsilon)} V(y)}{|\log\nu(x)|} \geq\frac
{1}{2t_0}, \qquad|x| \geq R_{\varepsilon}.
\]
\end{longlist}
\end{theorem}

\begin{theorem}[(Sufficient and necessary conditions for GSD)]\label{thmneciuc}
Let Assumptions~\ref{assassnu}--\ref{assasspinning} hold.
\begin{longlist}[(2)]
\item[(1)]
If
\[
\lim_{|x| \to\infty} \frac{V(x)}{|\log\nu(x)|} = \infty,
\]
then the semigroup $\{T_t\dvtx  t \geq0\}$ is GSD.

\item[(2)]
If the semigroup $\{T_t\dvtx  t \geq0\}$ is GSD, then for every
$\varepsilon\in
(0,1]$ we have
\[
\lim_{|x| \to\infty} \frac{\sup_{y \in B(x,\varepsilon)}
V(y)}{|\log\nu
(x)|} = \infty.
\]
\end{longlist}
\end{theorem}

\begin{remark}
By Theorem~\ref{thmiucvelgsd}, the limit condition in Theorem~\ref{thmneciuc}(2) is
necessary for IUC, and the condition in Theorem~\ref{thmneciuc}(1) is
sufficient for IUC,
whenever $p(t,\cdot)$ is bounded for all $t>0$. Similarly, the growth
condition in Theorem
\ref{thmsuffiuc}(2) is necessary for $t_0/2$-IUC, and the condition
in Theorem~\ref{thmsuffiuc}(1) is sufficient for $t_0$-IUC with $t_0 = 2(t_b \vee4/C_{13})$.
\end{remark}

The following result is intuitively clear, however, for the reader's
convenience we include a short
proof at the end of Section~\ref{subsecprel}.

\begin{proposition}[(Ultracontractivity)]\label{propultracontractivity}
Let Assumptions~\ref{assassnu}--\ref{assasspinning} hold, in
particular, let Assumption~\ref{assassdensity}
hold with $t_{b}>0$. Furthermore, suppose that:
\begin{longlist}[(2)]
\item[(1)]
there exists $t < t_{b}$ such that $\lim_{|x| \to0^{+}}p(t,x) =
\infty
$, and for all $s \in(0,t]$ we have that $p(s,x)
\geq p(s,y)$ whenever $|x| \leq|y|$;

\item[(2)]
there exist $x_0 \in\mathbf{R}^d$ and $\varepsilon> 0$ such that $V$ is
bounded from above in $B(x_0,\varepsilon)$.
\end{longlist}
Then\vspace*{1pt} for every $0<t<t_{b}$ for which condition (1) is satisfied, the
operator $T_{t/2}$ is not bounded from
$L^2(\mathbf{R}^d)$ to $L^{\infty}(\mathbf{R}^d)$. In particular,
the semigroup $\{T_t\dvtx  t \geq0\}$ is not ultracontractive.
\end{proposition}

Since $\varphi_0$ is bounded, IUC implies ultracontractivity. Hence,
the above result shows that the assumption
in assertion (2) of Theorem~\ref{thmiucvelgsd} is essential. This
means that there exists a class of random
processes whose Feynman--Kac semigroups are GSD but not IUC (even if
the potential grows to infinity at
infinity very quickly). Typical examples of L\'evy processes fitting
the above proposition include subordinate
Brownian motion with suitably slowly varying characteristic exponents
such as geometric stable processes. This
example will be discussed in more detail in Section~\ref{sec4.2}.

For the remainder of this subsection, we restrict attention to a
somewhat smaller class of potentials by imposing
more regularity. This will also be used in the next subsection.

\begin{assumption}
\label{assasscomp}
There exist $R>1$ and a constant $C_{14}=C_{14}(V)$ such that for every
$|x| > R$
%
\begin{equation}
\label{eqMq} V(y) \leq C_{14} V(x), \qquad y \in B(x,1)
\end{equation}
holds.
\end{assumption}

A straightforward consequence of the above theorems is the following result.

\begin{corollary}[(Borderline case)]\label{corborder}
Let Assumptions~\ref{assassnu}--\ref{assasscomp} hold, in particular,
let Assumption~\ref{assassdensity}
hold with $t_{b}>0$. Then we have the following:
\begin{longlist}[(2)]
\item[(1)]
The semigroup $\{T_t\dvtx  t\geq0\}$ is GSD if and only if
%
\begin{equation}
\label{eqcondcomp1} \lim_{|x| \rightarrow\infty} \frac{V(x)}{|\log\nu(x)|} = \infty.
\end{equation}
Moreover, condition (\ref{eqcondcomp1}) is necessary for IUC, and
sufficient whenever $p(t, \cdot)$ is
bounded for every fixed $t>0$.

\item[(2)]
The semigroup $\{T_t\dvtx  t\geq0\}$ is AGSD (or, equivalently, AIUC) if
and only if there exist a constant
$C_{15}$ and $R>0$
such that
%
\begin{equation}
\label{eqcondcomp2} \frac{V(x)}{|\log\nu(x)|} \geq C_{15}, \qquad|x| \geq R.
\end{equation}
Specifically, if (\ref{eqcondcomp2}) is satisfied, then $\{T_t\dvtx  t\geq
0\}$ always is $t_0$-GSD with $t_0=4/C_{15}$,
and it is $t_0$-IUC with $t_0 = 2(t_{b} \vee4/C_{15})$. If $\{T_t\dvtx
t\geq0\}$ is $t_0$-GSD, then (\ref{eqcondcomp2})
holds with constant $C_{15}=1/(2 C_{14} t_0)$. Similarly, $t_0$-IUC
implies (\ref{eqcondcomp2}) with $C_{15}=
1/(4 C_{14} t_0)$.
\end{longlist}
\end{corollary}

By the above results, we are now able to formally define borderline potentials.

\begin{definition}[(Borderline potential)]
\label{defborder}
Let Assumptions~\ref{assassnu}--\ref{assasspinning} hold. We call $V$
\emph{borderline potential} for (A)GSD/(A)IUC
of the semigroup $\{T_t\dvtx  t \geq0\}$ if there exist $t_0>0$ and $R>0$
such that
$t_0 V(x) = |\log\nu(x)|$, for every $x \in
B(0,R)^c$.
\end{definition}

Note that by Assumption~\ref{assassnu} the borderline potentials
always satisfy Assumption~\ref{assasscomp}. Also,
note that we speak of borderline potentials in the sense of equivalence
classes given by the definition above.
The examples of possible borderline potentials for different classes of
L\'evy processes are discussed in Section~\ref{subsecexres}.

\subsection{Probabilistic and variational interpretation of AGSD/AIUC}\label{sec2.5}
It was seen in the previous subsection that under Assumptions~\ref{assassnu}--\ref{assasspinning} AGSD/AIUC of
$\{T_t\dvtx  t \geq0\}$ depends only on the intensity of large jumps of the process
$(X_t)_{t \geq0}$. This means that whenever $\nu\neq0$,
the Gaussian and small jump parts of the process have no impact on
AGSD/AIUC. Indeed, the borderline
growth of $V$ is decided by the ratio $e^{-t_0V(x)}/\nu(x)$ for $x$
sufficiently
far away from the origin and some time point $t_0>0$. More precisely,
AGSD/AIUC occurs if and only if $e^{-t_0 V(x)}$
is uniformly dominated by the jump intensity $\nu(x)$ outside a bounded
set in $\mathbf{R}^d$. We note that although this
description gives a full picture of what AGSD/AIUC is in the case when
$\nu\neq0$, it does not help to understand
what is behind this property when the process is strictly diffusive,
that is, whenever $\nu= 0$. (In a sense, this
situation confirms that Brownian motion is an exceptional L\'evy
process and processes with jumps are the more
generic.) In this section, we discuss probabilistic and variational
descriptions of these properties.

It is straightforward that the condition on $V$ for $\{T_t\dvtx  t \geq0\}$ being
AGSD/\break AIUC is much weaker
than in the case of the Feynman--Kac semigroup for diffusions. This can
be explained by the following heuristic
interpretation. For our purposes here, it suffices to observe that the
effect of the potential on the distribution
of paths is a concurrence of killing at a rate of $e^{-\int_0^t
V_{+}(X_s)\,ds}$ and mass generation at a rate of
$e^{\int_0^t V_{-}(X_s)\,ds}$. When, however, $V(x) \to\infty$ as $|x|
\to\infty$, then outside a compact set
only the killing effect occurs and $\mathbf{E}^x[e^{-\int_0^t V(X_s)\,ds}]$
gives the probability of survival of the
process up to time $t$. The following characterization of AGSD/AIUC may
be used as a probabilistic definition
of these properties, valid for both our jump L\'evy processes and
Brownian motion. In fact, this property has a strong ergodic
flavor; compare also with \cite{bibDa,bibKP}.

\begin{proposition}
\label{propprobchar}
The semigroup $\{T_t\dvtx  t \geq0\}$ is AGSD/AIUC if and only if there
exist $t>0$,
a bounded nonempty Borel set $D \subset
\mathbf{R}^d$, and a constant $C_{16}=C_{16}(X,V,t)$ such that for
every Borel
set $A \subseteq\mathbf{R}^d$ we have
%
\begin{equation}
\label{eqprobchar} \mathbf{E}^x \bigl[e^{-\int_0^t V(X_s)\,ds}; X_t
\in A \bigr] \leq C_{16} \mathbf{E} ^x
\bigl[e^{-\int_0^t V(X_s)\,ds}; X_t \in D \bigr], \qquad x \in
\mathbf{R}^d.\hspace*{-30pt}
\end{equation}
\end{proposition}

(For a proof, see \cite{bibKL}, Corollary 4.1, Proposition 4.1, and
\cite{bibKS}, equations \mbox{(1.2)--(1.3)}.) Asymptotically,
the probability of survival of\vspace*{1pt} the process staying around the starting
point $x$ (far from the region $D$) is
approximately $e^{-t V(x)}$, while the probability of surviving by
escaping to a region $D$ with a lower killing
rate is $\mathbf{P}^x(X_t \in D)$. By using (\ref{eqprobchar}), it is
immediately seen that when $\{T_t\dvtx  t \geq0\}$ is AGSD/AIUC,
then the probability that the process under $V$ survives up to time~$t$
far from the location of $\inf V$ is
bounded by the probability that the process survives up to time $t$ and
is in some bounded region $D$, no matter
its starting point. It can be expected that the balance of the
competing effects in fact will be decided roughly
by the ratio $V(x)/|\log\mathbf{P}^x(X_t \in D)|$. Below we prove this
\mbox{intuition} and show that for a large class of
nondiffusive L\'evy processes the expression $|\log\mathbf{P}^x(X_t
\in D)|$
precisely determines the borderline
potential. Note that this expression does not give the borderline
potential for diffusions, however, it allows
to identify the leading order of the borderline growth which is known
to be quadratic~\cite{bibDS}. Some further
examples will be discussed below.

The following comparability condition will be used in Propositions~\ref{propAIUCproba}--\ref{propIUCproba}
only. It is partly satisfied under our previous assumptions and it
appears to be strongly related to Assumption
\ref{assassnu}. However, we are not aware of a general argument
showing a possible implication, and thus we
formulate it as an independent assumption.

\begin{assumption}
\label{assasspnu}
For every $t>0$, there is $R=R(t)>0$ such that
\[
\bigl|\log\nu(x)\bigr| \asymp C_{17} \bigl|\log p(t,x)\bigr| \asymp C_{18} \bigl|
\log\mathbf{P} ^x\bigl(X_{t} \in B(0,1)\bigr)\bigr|,
\qquad|x|>R
\]
with constants $C_{17}=C_{17}(X)$ and $C_{18}=C_{18}(X)$ (independent
of $t$).
\end{assumption}

The next two propositions are direct consequences of Assumption~\ref{assasspnu} and Theorems
\ref{thmsuffiuc}--\ref{thmneciuc}.

\begin{proposition}[(AGSD/AIUC probabilistically)]
\label{propAIUCproba}
Let Assumptions~\ref{assassnu}--\ref{assasspinning} and~\ref{assasspnu} be satisfied. Then the following hold:
\begin{longlist}[(2)]
\item[(1)]
If the semigroup $\{T_t\dvtx  t \geq0\}$ is $t_0$-GSD (or $t_0/2$-IUC),
then for
every \mbox{$0<\varepsilon\leq1$} there is $R \geq2$
such that
\begin{eqnarray}
\frac{\sup_{y \in B(x,\varepsilon)} V(y)}{|\log\mathbf
{P}^x(X_{t_0} \in
B(0,1))|} \geq\frac{1}{2 C^2_{17} C_{18} t_0} \quad\mbox{and} \quad
\frac{\sup_{y \in B(x,\varepsilon)} V(y)}{|\log p(t_0,x)|} \geq \frac{1}{2 C_{17} t_0},\nonumber
\\
\eqntext{|x| \geq R.}
\end{eqnarray}
Moreover, if also Assumption~\ref{assasscomp} holds, then $\sup_{y
\in
B(x,\varepsilon)} V(y)$ can be replaced by
$C_V V(x)$.

\item[(2)]
If there exist $t>0$, $R>0$ and a constant $C_{19}=C_{19}(X,V)$ such that
\[
\frac{V(x)}{|\log\mathbf{P}^x(X_{t} \in B(0,1))|} \geq\frac
{1}{C_{19} t} \quad\mbox{or}\quad\frac{V(x)}{|\log p(t,x)|}
\geq\frac{1}{C_{19} t}, \qquad|x|>R,
\]
then $\{T_t\dvtx  t \geq0\}$ is $t_0$-GSD with $t_0= 4 C_{18} C_{19} t$ and
$t_0$-IUC
with $t_0 = 2(t_{b} \vee4 C_{18} C_{19} t)$ or
$t_0$-GSD with $t_0= 4 C_{17} C_{19} t$ and $t_0$-IUC with $t_0 =
2(t_{b} \vee4 C_{17} C_{19} t)$, respectively.
\end{longlist}
\end{proposition}

\begin{proposition}[(GSD/IUC probabilistically)]
\label{propIUCproba}
Let Assumptions~\ref{assassnu}--\ref{assasspinning} and~\ref{assasspnu} be satisfied. Then the following hold:
\begin{longlist}[(2)]
\item[(1)]
If the semigroup $\{T_t\dvtx  t \geq0\}$ is GSD (or IUC), then for every
$t>0$ we have
\[
\lim_{|x| \to\infty} \frac{\sup_{y \in B(x,\varepsilon
)}V(y)}{|\log\mathbf{P}
^x(X_{t} \in B(0,1))|} = \lim_{|x| \to\infty}
\frac{\sup_{y \in B(x,\varepsilon
)}V(y)}{|\log
p(t,x)|} = \infty.
\]
When in addition also Assumption~\ref{assasscomp} holds, then $\sup_{y
\in B(x,\varepsilon)}V(y)$ may be replaced by
$V(x)$.

\item[(2)]
If there is $t>0$ such that
\[
\lim_{|x| \to\infty} \frac{V(x)}{|\log\mathbf{P}^x(X_{t} \in
B(0,1))|} = \infty\quad\mbox{or}\quad\lim
_{|x| \to\infty} \frac{V(x)}{|\log p(t,x)|} = \infty,
\]
then $\{T_t\dvtx  t \geq0\}$ is GSD. If, moreover, $p(t,\cdot)$ is bounded
for all
$t>0$, then any of these two conditions also implies
IUC.
\end{longlist}
\end{proposition}

Finally, we give another description of AGSD/AIUC. In order to do that,
we need to put one more condition
on the L\'evy measure.

\begin{assumption}
\label{assassnulog}
For every $R>0$, we have $\log\nu\in L^1(B(0,R)^c, \nu(x)\,dx)$.
\end{assumption}

Under Assumptions~\ref{assassnu}--\ref{assasscomp} and~\ref{assassnulog}, and for all $A \in\mathcal{B}(\mathbf{R}^d)$
such that $\operatorname{dist}(A,0)>0$ we define the functionals
%
\begin{eqnarray}
E^V_A(\nu) &=& \int_A V(x)\nu(x) \,dx, \label{energ}
\\
H_A(\nu) &=& -\int_A \nu(x)\log\nu(x) \,dx,\label{entrop}
\\
F^V_A(\nu) &=& E^V(\nu) -H_A(\nu). \label{free}
\end{eqnarray}
Note that
under Assumption~\ref{assassnulog} $F^V_A(\nu)$ is well defined. We
call the functional $E^V_A(\nu)$ \emph{energy},
$H_A(\nu)$ \emph{entropy} and $F^V_A(\nu)$ \emph{free energy} in set
$A$ for the given potential $V$ and L\'evy
measure $\nu(dx)$. Note that since $\nu(x)$ is the Radon--Nikod\'ym
derivative of the L\'evy measure with respect
to Lebesgue measure, $H_A(\nu)$ is in fact the relative entropy (or
Kullback--Leibler functional) of the L\'evy
measure with respect to Lebesgue measure. Then we have the following
characterization of AGSD/AIUC.

\begin{theorem}[(Characterization of AGSD/AIUC)]\label{thmcharacter}
Let Assumptions~\ref{assassnu}--\ref{assasscomp} and~\ref{assassnulog} hold. The potential $V$ is such
that the semigroup $\{T_t\dvtx  t \geq0\}$ is AGSD (or, equivalently, AIUC)
if and
only if there exists $t_0>0$ and $R>0$
such that for every Borel set $A \subset B(0,R)^c$ we have that
$F_A^{t_0 V}(\nu) \geq0$. Specifically, if
$\{T_t\dvtx  t \geq0\}$ is $t_0$-GSD, then $F_A^{2C_{14}t_0 V}(\nu) \geq
0$. If
$F_A^{t_0 V}(\nu) \geq0$, then $\{T_t\dvtx  t \geq0\}$
is $8 C_{14} t_0$-GSD.
\end{theorem}

Note that due to monotonicity of the free energy functional with
respect to potential $V$, the inequality
$F_A^{t_0 V}(\nu) \geq0$ implies $F_A^{t V}(\nu) \geq0$ for all $t
\geq t_0$. Furthermore, we have the
following variational result.

\begin{corollary}[(Variational principle for borderline potential)]\label{variational}
Let Assumptions~\ref{assassnu}--\ref{assasspinning} and~\ref{assassnulog} hold, and the jump intensity
$\nu$ and the potential $V$ be continuous functions. Then $V$ is the
borderline potential for AGSD/AIUC of the
semigroup $\{T_t\dvtx  t \geq0\}$ if and only if there exist $t_0>0$ and
$R>0$ such
that $F^{t_0 V}_A(\nu) = 0$ for all Borel
sets $A \subset B(0,R)^c$.
\end{corollary}

We note that similar energy and entropy functionals have been used in
\cite{bibGar} to determine heavy tailed
probability distributions with prescribed asymptotics, satisfying the
Fokker--Planck equation. Such optimization
methods are widely used, however, in our context it is derived and
rigorously justified by Theorem
\ref{thmcharacter}. Furthermore, the above variational problem can
also be considered in the reverse direction.
Roughly speaking, for a given sufficiently regular potential $V$ we may
be interested in finding the appropriate
L\'evy measures $\nu$ [i.e., L\'evy processes $(X_t)_{t \geq0}$] such
that the
corresponding free energy functional
$F^{t_0 V}_A(\nu)$ is minimized for some $t_0>0$, $R>0$ and every Borel
set $A \subset B(0,R)^c$.

\section{Proofs}\label{sec3}

\subsection{Preliminary results}\label{sec3.1}
\label{subsecprel}
Here, we recall some basic facts of potential theory for the
Feynman--Kac semigroup related to process
$(X_t)_{t \geq0}$ needed for our purposes, and show some technical
facts used in
proving our results concerning
intrinsic ultracontractivity and the eigenfunction estimates below. For
background, we refer to
\cite{bibCZ,bibBB1,bibBB2,bibCS,bibCS1,bibBlG}.

We adopt the convention that auxiliary constants appearing in proofs
may change their values from one
use to another (possibly from line to line). However, if necessary, we
write $C, C^{(1)}, C^{(2)}, \ldots$
to distinguish them. Recall that, in contrast, constants appearing in
the statements of theorems,
propositions and lemmas are fixed throughout the paper and can be
tracked in the proofs.

Denote by
\[
e_V(t):= e_V(t) (\omega) = e^{-\int_0^t V(X_s(\omega)) \,ds}, \qquad
t>0.
\]
The Feynman--Kac-functional for the L\'evy process $(X_t)_{t \geq0}$ for
potential $V$. By standard arguments
based on Khasminskii's lemma (see \cite{bibCZ}, Proposition~3.8,
and \cite{bibLHB}), there are constants $C_{21}=C_{21}(X,V)$
and $C_{22}=C_{22}(X,V)$ such that
%
\begin{equation}
\label{eqkhas} \sup_{x \in\mathbf{R}^d} \mathbf{E}^x
\bigl[e_V(t)\bigr] \leq\sup_{x \in
\mathbf{R}^d} \mathbf{E}
^x\bigl[e_{-V_{-}}(t)\bigr] \leq C_{21}
e^{C_{22}t}, \qquad t>0.
\end{equation}
Recall that $\tau_D = \inf\{t > 0\dvtx    X_t \notin D\}$ denotes the first
exit time of the process from
the set $D$. The potential operator for the semigroup $\{T_t\dvtx  t \geq
0\} $ is defined by
\[
G^V f(x) = \int_0^\infty
T_t f(x) \,dt = \mathbf{E}^x \biggl[\int
_0^\infty e_V(t) f(X_t) \,dt
\biggr]
\]
for nonnegative or bounded Borel functions $f$ on $\mathbf{R}^d$,
while the
$V$-Green operator for an open set $D$
is given by
\[
G^V_D f(x) = \int_0^\infty
\mathbf{E}^x \bigl[t<\tau_D; e_V(t)
f(X_t) \bigr] \,dt = \mathbf{E}^x \biggl[\int
_0^{\tau_D} e_V(t) f(X_t) \,dt
\biggr]
\]
for nonnegative or bounded Borel functions $f$ on $D$.

It can be seen directly that if $D \subset\mathbf{R}^d$ is a nonempty bounded
open set and $V$ is a nonnegative and
not identically zero potential on $D$, then for all $x \in D$ we have
%
\begin{equation}
\label{eqgsVest} \Bigl(1-\exp\Bigl(-\sup_{y \in D} V(y)\Bigr)
\Bigr) \frac{\mathbf{P}^x(\tau
_D > 1)}{\sup_{y \in D} V(y)} \leq G^V_D\mathbf{1}(x) \leq
\frac{1}{\inf_{y \in D} V(y)}.
\end{equation}
Here, we use the convention that $1/\infty= 0$ and $1/0^{+}=\infty$.

Below we often use the fact that for all bounded Borel sets $D \subset
\mathbf{R}^d$ and $x \in D$ we have $\mathbf{E}^x [\tau_D]
\leq\mathbf{E}^x [\tau_{B(x,\operatorname{diam}D)}] = \mathbf{E}^0
[\tau_{B(0,\operatorname{diam}D)}] < \infty
$. Furthermore, when $D' \subset\mathbf{R}^d$ is an
open set, $D \subset D'$ is open and bounded and $f$ is a nonnegative
or bounded Borel function on $D'$, then by
the strong Markov property, it follows for every $x \in D$ that
%
\begin{equation}
\label{eqpot1} G^V_{D'}f(x) = G^V_D
f(x) + \mathbf{E}^x \bigl[X_{\tau_D} \in D'
\setminus D; e_V(\tau _D)G^V_{D'}f(X_{\tau_D})
\bigr].
\end{equation}

A Borel function $f$ on $\mathbf{R}^d$ is called $(X,V)$-harmonic in
an open
set $D \subset\mathbf{R}^d$ if
%
\begin{eqnarray}
\label{defharm} f(x) & =& \mathbf{E}^x \bigl[\tau_U <
\infty; e_V(\tau_U) f(X_{\tau
_U}) \bigr], \qquad
x \in U,
\end{eqnarray}
for every open set $U$ with $\overline{U}$ contained in $D$, and it is
called regular \mbox{$(X,V)$-}har\-mo\-nic in $D$
if (\ref{defharm}) holds for $U=D$. By the strong Markov property,
every regular
$(X,V)$-harmonic function in $D$ is $(X,V)$-harmonic in $D$. We always
assume that the expectation in (\ref{defharm})
is absolutely convergent.

The following uniform estimates for local suprema of $(X,V)$-harmonic
functions are an important ingredient
in proving AGSD/GSD and eigenfunction bounds. Under Assumptions~\ref{assassnu}--\ref{assassbhi}, they directly
follow from the more general results in \cite{bibBKK}.

\begin{lemma}
\label{lmbhiold}
Let Assumptions~\ref{assassnu}(1),~\ref{assassdensity} and~\ref{assassbhi} be satisfied. Then for every
$0<r<p<q<R\leq1$ there exists a constant $C_{23}=C_{23}(X,r,p,q,R)$
such that for any $V \in\mathcal{K}^X_{\operatorname{loc}}$,
$V \geq0$ on $B(x_0,R)$, and every nonnegative function $f$ on
$\mathbf{R}^d$
that is $(X,V)$-regular harmonic on
$B(x_0,R)$, we have
\[
f(y) \asymp C_{23} G^V_{B(x_0,p)}\mathbf{1}(y) \int
_{B(x_0,q)^c}f(z)\nu (z-x_0)\,dz, \qquad|y-x_0|<r.
\]
\end{lemma}

\begin{pf}
Under the assumptions of the lemma, Assumptions A--D in \cite{bibBKK}
are satisfied. Specifically, since
$(X_t)_{t \geq0}$ is a symmetric L\'evy process satisfying the strong Feller
property, Assumptions A~and~B hold
directly, while Assumption~C is a consequence of our Assumption~\ref{assassnu}(1), and our Assumption
\ref{assassbhi} is just Assumption~D (for details of their
verification, see \cite{bibBKK}, Example 5.5).
Thus, the above estimates hold for $V \equiv0$ as a consequence of
\cite{bibBKK}, Lemma 3.2 and Theorem 3.4,
for the set $D=B(x_0,R)$. By space homogeneity of $(X_t)_{t \geq0}$, the
constant $C_{23}$ does not depend on the
specific choice of $x_0$. Similar estimates for an arbitrary $V \in
\mathcal{K}
^X_{\operatorname{loc}}$, $V \geq0$ on $B(x_0,R)$,
that is, for the subprocess of $(X_t)_{t \geq0}$ given by the multiplicative
functional $M_t=e_V(t)$, follow from the
latter with the same constant~$C_{23}$ (independent of $V$) by the
argument in \cite{bibBKK}, Example 5.9.
\end{pf}

Note that it is crucial below that $C_{23}$ in the above bounds does
not depend on the (local behavior of)
the potential $V$. It is also essential for our further applications
that due to space-homogeneity of the
process $(X_t)_{t \geq0}$ the constant is independent of the location
of the
ball $B(x_0,R)$ in space. This is also
the reason why we cannot consider in this paper more general Markov
processes that are not space-homogeneous.
In fact, Lemma~\ref{lmbhiold} will be used below in the following form
which is sufficiently general and
suitable for our purposes.

\begin{corollary}
\label{lmbhi}
Let Assumptions~\ref{assassnu}(1),~\ref{assassdensity} and~\ref{assassbhi} be satisfied.
Then there exists a constant $C_{24}=C_{24}(X)$ such that for any $V
\in\mathcal{K}^X_{\operatorname{loc}}$,
$V \geq0$ on $B(x_0,1)$, and every nonnegative function $f$ on
$\mathbf{R}^d$
that is $(X,V)$-regular harmonic on
$B(x_0,1)$, we have
\[
f(y) \asymp C_{24} G^V_{B(x_0,1)}\mathbf{1}(y) \int
_{B(x_0,3/4)^c}f(z)\nu (z-x_0)\,dz, \qquad|y-x_0|<
\frac{1} 2.
\]
\end{corollary}

\begin{pf}
By taking $r=1/2$, $p=5/8$, $q=3/4$ and $R=1$ in Lemma~\ref{lmbhiold},
we clearly have
\[
f(y) \asymp C_{23} G^V_{B(x_0,5/8)}\mathbf{1}(y) \int
_{B(x_0,3/4)^c}f(z)\nu(z-x_0)\,dz, \qquad|y-x_0|<
\frac{1} 2.
\]
Thus,\vspace*{1pt} it suffices to see that $G^V_{B(x_0,1)}\mathbf{1}(y) \leq C
G^V_{B(x_0,5/8)}\mathbf{1}(y)$, $y \in B(x_0,1/2)$, with a
constant $C=C(X)$ (independent of $V$ and $x_0$). By formula (\ref
{eqpot1}) for $D'=B(x_0,1)$,
$D=B(x_0,5/8)$ and $f=\mathbf{1}$, and by the fact that
\[
\sup_{y \in B(x_0,1)}
G^V_{B(x_0,1)}\mathbf{1}(y) = C^{(1)} <\infty
\]
with $C^{(1)}$ independent of $V$ and $x_0$,
we have
\[
G^V_{B(x_0,1)}\mathbf{1}(y) \leq G^V_{B(x_0,5/8)}
\mathbf{1}(y) + C^{(1)} \mathbf{E}^y\bigl[e_V(
\tau _{B(x_0,5/8)})\bigr], \qquad|y-x_0|<\tfrac{1} 2.
\]
Let now
\[
g(y) = \cases{\displaystyle\mathbf{E}^y\bigl[e_V(
\tau_{B(x_0,5/8)})\bigr], &\quad if $y \in B(x_0,5/8)$,
\vspace*{2pt}\cr
\mathbf{1}, &\quad if $y \notin B(x_0,5/8)$.}
\]
By applying Lemma~\ref{lmbhiold} for $g$ with $r=1/2$, $p=17/32$,
$q=9/16$ and $R=5/8$, we conclude that
\begin{eqnarray}
\mathbf{E}^y\bigl[e_V(\tau_{B(x_0,5/8)})\bigr] \leq
C_{23} G^V_{B(x_0,17/32)}\mathbf{1}(y) \int
_{B(0,9/16)^c} \nu(z)\,dz \leq C G^V_{B(x_0,5/8)}
\mathbf{1}(y),\nonumber
\\
\eqntext{\displaystyle |y-x_0|< \tfrac{1} 2}
\end{eqnarray}
with constant $C=C(X)$, independent of $V$ and $x_0$.
\end{pf}

In fact, in order to obtain the above corollary it suffices to prove
Lemma~\ref{lmbhiold} only for two fixed
sets of parameters $r, p, q, R$. Therefore, it would be enough to have
Assumptions~\ref{assassnu}(1) and
\ref{assassbhi} in place only for some specially chosen, sufficiently
small $r>0$ and $p, q >0$, respectively.
However, this approach requires a detailed analysis of constants
appearing in \cite{bibBKK} and causes further
technical difficulties (note that the parameters $r, p, q$ in Lemma
\ref
{lmbhiold} do not correspond directly
to $r$ and $p, q$ in the assumptions). Since the
general Assumptions~\ref{assassnu}(1) and~\ref{assassbhi} are not
restrictive for our further results, we
included a general version of Lemma~\ref{lmbhiold}.

The following auxiliary results will also be used later.

\begin{lemma}
\label{lmtechiu}
Let $D \subset\mathbf{R}^d$ be an arbitrary open set and $V$ be an $X$-Kato
class potential such that $V \geq0$
on $D$. Then there are constants $C_{25}=C_{25}(X,V,t)$ and
$C_{26}=C_{26}(X,V,t)$ such that for every $t>0$
we have
\begin{longlist}[(2)]
\item[(1)]
$\mathbf{E}^x [\frac{t}{2} \geq\tau_D; e_V(t) ] \leq
C_{25}
\mathbf{E}^x [e_V(\tau_D)T_{t/2}\mathbf{1}(X_{\tau_D}) ]$;

\item[(2)]
$\mathbf{E}^x [\frac{t}{2}<\tau_D; e_V(t) ] \leq C_{26}
G_D^V\mathbf{1}(x)   \sup_{y \in D}T_{t/2}\mathbf{1}(y)$, $x \in D$.
\end{longlist}
\end{lemma}

\begin{pf}
The proof of (1) and (2) with the expression on the right-hand side
$G_D^V\mathbf{1}(x)  \sup_{y \in D}T_{t/2}\mathbf{1}(y)$ replaced
by $\mathbf{E}^x [\frac
{t}{4} <
\tau_D; e_V (\frac{t}{4} ) ]  \sup_{y \in
D}T_{3t/4}\mathbf{1}
(y)$ runs in the same way as in \cite{bibKL}, Lemma 4.3.
We complete the proof of (2) by the simple observation that
\begin{eqnarray}
\mathbf{E}^x \biggl[\frac{t}{4} < \tau_D;
e_V \biggl(\frac
{t}{4} \biggr) \biggr] \leq
\frac{4} t \mathbf{E}^x \biggl[\frac{t}{4} <
\tau_D; \int_0^{t/4}
e^{-\int
_0^v V(X_s)\,ds}\,dv \biggr] \leq\frac{4} t G_D^V
\mathbf{1}(x),\nonumber
\\
\eqntext{x \in D}
\end{eqnarray}
and
\begin{eqnarray*}
T_{3t/4}\mathbf{1}(y) &=& \mathbf{E}^y \biggl[e_V
\biggl(\frac
{t}{2} \biggr)\mathbf{E}^{X_{t/2}} \biggl[
e_V \biggl(\frac{t}{4} \biggr) \biggr] \biggr]
\\
&\leq&
T_{t/2}(y) \sup_{z \in\mathbf{R}^d} \mathbf{E}^z
\biggl[e_V \biggl(\frac
{t}{4} \biggr) \biggr]
\leq
C_{V,t} T_{t/2}(y), \qquad y \in D.
\end{eqnarray*}\upqed
\end{pf}

A short proof of the following fact was communicated to us by M. Kwa\'snicki.

\begin{lemma}
\label{lmcompden}
Let $(X_t)_{t \geq0}$ be a L\'evy process with transition densities
$p(t,x,y)=p(t,y-x)$ such that for some $t>0$ and
all $s \in(0,t]$ we have $p(s,x) \leq p(s,y)$ whenever $|x| \geq|y|$.
Then for every bounded open set
$D \subset\mathbf{R}^d$ and $r>0$ such that $ \{y \in D\dvtx
\operatorname{dist}(y,\partial
D) \geq r \} \neq\varnothing$
there is a constant $C_{27}=C_{27}(r)$ such that
\[
p_D(t,x,y) \geq p(t,y-x) - C_{27}, \qquad x, y \in D,
 \operatorname{dist}(y, \partial D) \geq r.
\]
\end{lemma}

\begin{pf}
Observe that for every $|z| \geq r$ and $s \in(0,t]$ we have
\[
\bigl|B(0,r)\bigr| p(s,z) \leq\int_{B(0,r)} p(s,w)\,dw \leq1.
\]
Thus, by Hunt's formula (\ref{eqHuntF}) we get
\begin{eqnarray*}
p_D(t,x,y) &=& p(t,y-x) - \mathbf{E}^x \bigl[
\tau_D < t; p(t-\tau_D, y-X_{\tau
_D}) \bigr]
\\
&\geq& p(t,y-x) - C_{27}
\end{eqnarray*}
for all $x, y \in D$ such that $\operatorname{dist}(y, \partial D)
\geq r$, with
$C_{27} = |B(0,r)|^{-1}$.
\end{pf}

\subsection{Jump estimates}\label{sec3.2}
\label{subsecjumpest}
For our purposes below, we will need to control jumps between some
carefully chosen regions.
Let $n, k \in\mathbf{N}$, $n, k \geq n_0 \geq2$ (with $n_0$ to be chosen
below), and define
\begin{eqnarray*}
D_n &:=& \bigl\{x \in\mathbf{R}^d\dvtx  n-2 < |x| \bigr\},
\qquad n \geq n_0+2,
\\
D_{n_0} &=& D_{n_0+1}:= \mathbf{R}^d,
\\
R_k &:=& \bigl\{x \in\mathbf{R}^d\dvtx  k-1 < |x| \leq k
\bigr\}, \qquad k \geq n_0+2,
\\
R_{n_0} &:=& \bigl\{x \in\mathbf{R}^d\dvtx  |x| \leq
n_0 \bigr\},
\\
R_{n_0+1}&:=& \bigl\{x \in\mathbf{R}^d\dvtx
|x| \leq n_0+1 \bigr\}.
\end{eqnarray*}
We will use the two stopping times
\begin{eqnarray*}
\tau_n &=&\tau_{D_n}:= \inf \{t \geq0\dvtx  X_t
\notin D_n \},
\\
\sigma_k &=&\sigma_{R_k}:=\inf \{t \geq0\dvtx
X_t \in R_k \}.
\end{eqnarray*}
Note that $\tau_{n_0} = \tau_{n_0+1} = \infty$. In the events in which
we are interested, the process
jumps from the complement of a ball $D_n$ to a smaller shell $R_k$,
which we will refer to as admissible
jumps. We define for $k \geq n_0$, $n \geq k+2$ and $t>0$ the events
\begin{eqnarray*}
 S(n,k,1,t)&=& \{X_{\tau_n} \in R_k, \sigma_k,
< t \},
\\
S(n,k,l,t)&=& \bigcup_{p=k+2}^{n-2}
S(n,p,l-1,t) \cap S(p,k,1,t), \qquad l >1.
\end{eqnarray*}
The first corresponds to the event that the process arrives in shell
$R_k$ before time $t$ in just
one jump after exiting $D_n$. The second event is defined inductively.
Let $k+2 \leq p \leq n-2$.
The event $S(n,p,1,t) \cap S(p,k,1,t)$ means that the process jumps to
shell $R_p$ on leaving $D_n$
and then again jumps to shell $R_k$ on leaving~$D_p$, and all this
occurs before time $t$. Note that
the process may go elsewhere after arriving in $R_p$ but the events
which we are constructing only
take account of admissible jumps, that is, those that are oriented to
the origin through jumps into the
shells~$R_k$. Thus, the event $S(n,k,l,t)$ corresponds to the process
arriving in shell $R_k$ from $D_n$
through $l$ admissible jumps before time $t$. This scheme of keeping
track of the so defined jumps has
been first devised in \cite{bibBK} and used in \cite{bibKS}. Here, we
also partially
adopt the notation of \cite{bibKS}.

The following technical lemma will be needed below.

\begin{lemma}
\label{lmlevykernel}
Let Assumption~\ref{assassnu}(1)--(2) be satisfied, and take $n, k
\in\mathbf{N}$ such that $n-2 > k \geq n_0$.
Then the following hold:
\begin{longlist}[(2)]
\item[(1)] There is a constant $C_{28}=C_{28}(X) \geq1$ (independent
of $n$ and $k$) such that
\[
\int_{R_k} \nu(z-y)\,dz \leq C_{28} \int
_{R_k} \nu(z-x)\,dz, \qquad x \in R_n, y \in
D_n.
\]
\item[(2)] For any $m \in\mathbf{N}$, there is a constant
$C_{29}=C_{29}(X, m)
\geq1$ (independent of $k$) such that
\[
\int_{R_{k+m} \cap \{z\dvtx   |y-z|>1/2 \}} \nu(z-y)\,dz \leq C_{29} \int
_{R_k} \nu(z-y)\,dz, \qquad |y| \geq k+1.
\]
\end{longlist}
\end{lemma}

\begin{pf}
First, we prove (1). By rotation symmetry of $R_k$ and conditions
\mbox{(1)--(2)} of Assumption~\ref{assassnu},
we deduce directly that
\[
\int_{R_k} \nu(z-y)\,dz \leq C^2_1
\int_{R_k} \nu(z-x)\,dz, \qquad x \in R_n, n-2 <
|y| \leq|x|
\]
and
\[
\int_{R_k} \nu(z-y)\,dz \leq C_2 \int
_{R_k} \nu(z-x)\,dz, \qquad x \in R_n, |x|
\leq|y|,
\]
respectively. Thus, (1) follows. Consider assertion (2) of the lemma.
Define the dilations $S_{k,m}(w)
= ((k+m-1)/(k-1))w$, $m \in\mathbf{N}$. Since $R_{k+m} \subset S_{k,m}(R_k)$,
it suffices to prove (2) for
$R_{k+m}$ replaced by $S_{k,m}(R_k)$. By changing variables, we obtain
\begin{eqnarray*}
&& \int_{S_{k,m}(R_k) \cap \{z\dvtx   |y-z|>1/2 \}}  \nu(z-y)\,dz
\\
&&\qquad = \biggl(\frac{k+m-1}{k-1} \biggr)^d \int_{R_k \cap S^{-1}_{k,m}(
\{z\dvtx
|y-z|>1/2 \})}
\nu \biggl(\frac{k+m-1}{k-1}w-y \biggr)\,dw
\end{eqnarray*}
for all $|y| \geq k+1$. Since $\llvert \frac{k+m-1}{k-1}w - w\rrvert
\leq
2m$, by Assumption
\ref{assassnu}(1) we have
\[
\nu \biggl(\frac{k+m-1}{k-1}w-y \biggr) \leq C^{(1)} \nu (w-y )
\]
with $w \in R_k \cap S^{-1}_{k,m}( \{z\dvtx   |y-z|>1/2 \})$ and
constant $C^{(1)} = C^{(1)}(X,m)$.
Hence,
\begin{eqnarray}
\int_{S_{k,m}(R_k) \cap \{z\dvtx   |y-z|>1/2 \}}  \nu(z-y)\,dz \leq C^{(1)}
(m+1)^d \int_{R_k} \nu (w-y )\,dw,\nonumber
\\
\eqntext{|y| \geq k+1,}
\end{eqnarray}
which completes the proof of the lemma.
\end{pf}

The next two lemmas are key tools to our further considerations. Lemma
\ref{lmjumpest1}
builds on \cite{bibKS}, Lemma 5.7,
however, our argument is based on a completely new approach which
combines sharp uniform upper estimates for local maxima of
$(X,V)$-harmonic functions (Corollary~\ref{lmbhi})
with an inductive procedure which substantially uses Assumption~\ref{assassnu}(3). We recall that a sufficiently
general version of this uniform estimate necessary for our purposes in
this paper was proved only recently in
\cite{bibBKK}. The second Lemma~\ref{lmjumpest2} is a
corollary of Lemma~\ref{lmjumpest1} and Assumption~\ref{assassnu}.
Notice that it is crucial for the applications
below that the constants $C_{30}$ and $C_{31}$ in these lemmas are
independent of $t$, unlike in \cite{bibKS}. This
allows us to use them in proving estimates of $\lambda$-subaveraging
functions and, in consequence, the bounds on
the eigenfunctions. Both proofs below clearly show the significance of
condition (3) in Assumption~\ref{assassnu}.

We note for later use that under condition (1) in Assumption~\ref{assassnu} we have $\mathbf{P}^x(\tau_n < t) > 0$ and
$\mathbf{P}^x(\sigma_k < t) > 0$, for all $n-2 \geq k \geq n_0$,
$n-1<|x|\leq n$
and~$t>0$.

\begin{lemma}
\label{lmjumpest1}
Let Assumptions~\ref{assassnu}--\ref{assassbhi} hold, and $n, k \in
\mathbf{N}
$ be such that $n-2\geq k
\geq n_0$. Then there is a constant $C_{30}=C_{30}(X)$ and $\theta
_0=\theta_0(X) \geq1$ such that for every $t>0$,
for all $n-1 < |x| \leq n$ and $\theta>\theta_0$ we have
\[
\mathbf{E}^x \bigl[\tau_n < t, X_{\tau_n} \in
R_k; e^{-\theta\tau
_n} \bigr] \leq\frac{C_{30}}{\theta} \int
_{R_k} \nu(y-x)\,dy.
\]
\end{lemma}

\begin{pf}
First, we assume that $n> k+ 2$ and $\theta>0$ is arbitrary. Note that
$\operatorname{dist}(D_n, R_k) \geq1$.
For $r>n-2$ denote $\tau_{n,r}:=\tau_{D_n \cap B(0,r)}$. Using the
Ikeda--Watanabe formula
\cite{bibIW}, Theorem 1, we have
\begin{eqnarray*}
&& \mathbf{E}^x \bigl[\tau_{n,r} < t, X_{\tau_{n,r}} \in
R_k; e^{-\theta\tau
_{n,r}} \bigr]
\\
&&\qquad \leq\int_{D_n \cap B(0,r)}  \int_0^{\infty}
e^{-\theta s} p_{D_n
\cap
B(0,r)}(s,x,y) \int_{R_k}
\nu(z-y)\,dz\,ds\,dy
\end{eqnarray*}
and, consequently, by Lemma~\ref{lmlevykernel}(1) and Fubini's
theorem we get
\begin{eqnarray*}
&& \mathbf{E}^x \bigl[\tau_{n,r} < t, X_{\tau_{n,r}} \in
R_k; e^{-\theta\tau
_{n,r}} \bigr]
\\
&&\qquad \leq C \int_0^{\infty}
e^{-\theta s} \,ds \int_{R_k} \nu(z-x)\,dz
= \frac {C}{\theta
} \int_{R_k} \nu(z-x)\,dz
\end{eqnarray*}
with constant $C=C(X)$. To complete the proof in this case, it suffices
to show that
\[
\mathbf{E}^x \bigl[\tau_{n,r} < t, X_{\tau_{n,r}} \in
R_k; e^{-\theta\tau
_{n,r}} \bigr] \to \mathbf{E}^x \bigl[
\tau_{n} < t, X_{\tau_{n}} \in R_k;
e^{-\theta
\tau_{n}} \bigr]
\]
as $r \to\infty$. Since $\tau_{n,r} = \tau_{n}$ when $X_{\tau_{n,r}}
\in R_k$, we have
\begin{eqnarray*}
0 &\leq&\mathbf{E}^x \bigl[\tau_{n} < t, X_{\tau_{n}}
\in R_k; e^{-\theta\tau_{n}} \bigr] - \mathbf{E}^x \bigl[
\tau_{n,r} < t, X_{\tau_{n,r}} \in R_k;
e^{-\theta\tau
_{n,r}} \bigr]
\\
& =& \mathbf{E}^x \bigl[\tau_{n,r} < \tau_{n} <
t, X_{\tau_{n}} \in R_k, X_{\tau
_{n,r}} \in
B(0,r)^c; e^{-\theta\tau_{n}} \bigr]
\\
& \leq&\mathbf{P}^x \bigl(\tau_{n,r} < t, X_{\tau_{n,r}}
\in B(0,r)^c \bigr)
\\
& \leq&\mathbf{P}^x (\tau_{B(0,r)} < t ) \leq\mathbf
{P}^0(\tau_{B(0,r/2)} < t)
\end{eqnarray*}
for sufficiently large $r$. Clearly, $\mathbf{P}^0(\tau_{B(0,r/2)} <
t) \to0$
as $r \to\infty$, $t>0$,
and the claimed convergence follows.

Now consider the case $n=k+2$. Denote $B_y=B(y,1)$ and
\[
f(y) = \cases{ \displaystyle \mathbf{E}^y \bigl[
\tau_n < \infty, X_{\tau_n} \in R_k;
e^{-\theta\tau_n} \bigr], &\quad if $y \in D_n$,
\vspace*{2pt}\cr
\mathbf{1}_{R_k}(y), &\quad if $y \notin D_n$.}
\]
%
Since $f$ is an $(X,\theta)$-regular harmonic function in $D_n$, we have
\[
f(z) = \mathbf{E}^z \bigl[e^{-\theta\tau_{B_y}} f(X_{\tau
_{B_y}})
\bigr], \qquad z \in B_y, |y|>n-1.
\]
We will show that there is a constant $C=C(X)$ and $\theta_0\geq1$ such
that for every $\theta\geq\theta_0$ and $l \in\mathbf{N}$
%
\begin{eqnarray}\label{eqauxest1}
f(y) \leq\frac{C}{\theta} \Biggl(\sum
_{i=1}^{l}\frac{1}{2^i} \int
_{R_k} \nu(z-y) \,dz + \frac{1}{2^l}\int
_{|z|>n+1} \nu(z-y)f(z)\,dz \Biggr),
\nonumber\\[-16pt]\\[-2pt]
\eqntext{n-1 < |y| \leq n.}
\end{eqnarray}
If this holds, then by taking the limit $l \to\infty$ the required
bound follows.

By Corollary~\ref{lmbhi}, we have
\begin{eqnarray}
f(z) \leq C_{24} G^{\theta}_{B_y}\mathbf{1}(z) \int
_{B(y,1/2)^c} f(w) \nu (w-y)\,dw,\nonumber
\\
\eqntext{|z-y|<1/2,  n-1 < |y|}
\end{eqnarray}
and, since $G^{\theta}_{B_y} \mathbf{1}(z) \leq1/\theta$, we obtain
%
\begin{eqnarray}\label{eqauxest0}
f(y) \leq\frac{2 C_{24}}{2 \theta} \biggl(\int_{R_k}
\nu(z-y)\,dz + \int_{{|z-y|>1/2,|z|>n-2}} \nu(z-y)f(z)\,dz \biggr),
\nonumber\\[-3pt]\\[-16pt]
\eqntext{n-1 <|y|.}
\end{eqnarray}
Moreover, $0 \leq f \leq1$ and a direct application of Lemma~\ref{lmlevykernel}(2)
to each of the three integrals (recall that $k=n-2$)
\[
\int_{R_{k+m} \cap \{z\dvtx   |z-y|>1/2 \} } \nu(z-y)\,dz, \qquad m = 1, 2, 3,
\]
separately gives that
%
\begin{eqnarray}\label{eqauxest2}
f(y) \leq\frac{C}{2 \theta} \biggl(\int_{R_k}
\nu(z-y)\,dz + \int_{{|z-y|>1/2,|z|>n+1}} \nu(z-y)f(z)\,dz \biggr),
\nonumber\\[-3pt]\\[-16pt]
\eqntext{n-1 <|y|}
\end{eqnarray}
with constant $C=C(X)$. In particular, (\ref{eqauxest1}) holds for
$l=1$ and arbitrary $\theta>0$.

Next, we use induction. Let $\theta_0:=C C_3 \vee1$ and suppose that
(\ref{eqauxest1}) is true for $l-1 \in\mathbf{N}$ with constant $C$ and
$\theta\geq\theta_0$, where $C$ is the constant in (\ref
{eqauxest2}). By the induction hypothesis and (\ref{eqauxest2}), we
have for $n-1 < |y| \leq n$ and $\theta\geq\theta_0$
\begin{eqnarray*}
f(y) & \leq&\frac{C}{\theta} \Biggl(\sum_{i=1}^{l-1}
\frac{1}{2^i} \int_{R_k} \nu(z-y) \,dz + \frac{1}{2^{l-1}}
\int_{|z|>n+1} \nu(z-y)f(z)\,dz \Biggr)
\\
& \leq& \frac{C}{\theta}\sum_{i=1}^{l-1}
\frac{1}{2^i} \int_{R_k} \nu(z-y) \,dz
\\
&&{} + \biggl(\frac{C}{\theta} \biggr)^2 \frac{1}{2^{l}}\int
_{|z|>n+1} \nu (z-y)\int_{R_k} \nu(w-z)\,dw\,dz
\\
&&{}+ \biggl(\frac{C}{\theta} \biggr)^2 \frac{1}{2^{l}}\int
_{|z|>n+1} \nu(z-y) \int_{{|w-z|>1/2,|w|>n+1}}
\nu(w-z)f(w)\,dw\,dz.
\end{eqnarray*}
An application of Fubini's theorem and Assumption~\ref{assassnu}(3)
to the last two terms in the
sum on the right-hand side above gives
\begin{eqnarray}
f(y) &\leq&\frac{C}{\theta} \sum_{i=1}^{l}
\frac{1}{2^i} \int_{R_k} \nu (w-y) \,dw +
\frac{C}{\theta} \frac{1}{2^{l}} \int_{{|w|>n+1}}
\nu(w-y)f(w)\,dw,\nonumber
\\[-2pt]
\eqntext{n-1 < |y| \leq n.}
\end{eqnarray}
\end{pf}

\begin{lemma}
\label{lmjumpest2}
Let Assumptions~\ref{assassnu}--\ref{assasspinning} hold. Moreover,
suppose that there is a
nondecreasing sequence $g_n \to\infty$ as $n \to\infty$ and $n_0
\in
\mathbf{N}$ large enough such that
for $n \geq n_0$ we have
\[
1< 2 \theta_0 \leq g_{n} \leq\inf_{|y|\geq n}
V(y) \quad\mbox{and} \quad 4 C_3 C_{28}
C_{30} \leq g_{n_0},
\]
where $C_{28}$, $C_{30}$ are constants, and $\theta_0$ is the parameter
from Lemmas~\ref{lmlevykernel}(1)
and~\ref{lmjumpest1}, respectively. Then for $n-1 < |x| \leq n$, $n_0
\leq k \leq n-2$, $n, k, l \in\mathbf{N}$,
it follows that
%
\begin{eqnarray}\label{eqjumpest2}
\mathbf{E}^x \bigl[S(n,k,l,t); e^{-(1/2) \int_0^{\sigma_k}
V(X_s)\,ds}
\bigr] \leq \frac{C_{31}}{2^l g_{n-2}} \int_{R_k} \nu(y-x)\,dy
\nonumber\\[-9pt]\\[-9pt]
\eqntext{\mbox{with } C_{31} = 4 C_{30}.}
\end{eqnarray}
\end{lemma}

\begin{pf}
We use induction in $l \in\mathbf{N}$. Let $l=1$. Since we have $S(n,k,1,t) =
\{
X_{\tau_n} \in R_k, \sigma_k <t  \}$
and $\tau_n = \sigma_k$ for $X_{\tau_n} \in R_k$, we obtain by Lemma
\ref{lmjumpest1}
\begin{eqnarray*}
\mathbf{E}^x \bigl[S(n,k,1,t); e^{-(1/2) \int_0^{\sigma_k}
V(X_s)\,ds} \bigr] & \leq&
\mathbf{E}^x \bigl[X_{\tau_n} \in R_k,
\tau_n < t; e^{-\tau_n g_{n-2}/2 } \bigr]
\\
& \leq&\frac
{C_{31}}{2g_{n-2}} \int_{R_k} \nu(y-x)\,dy.
\end{eqnarray*}

Let now $l \geq2$ and suppose that the bound (\ref{eqjumpest2}) holds
for $1, 2, \ldots, l-1$ and all $n, k$ as in the statement of the lemma.
The strong Markov
property gives
\begin{eqnarray*}
&& \mathbf{E}^x  \bigl[S(n,k,l,t); e^{-(1/2) \int_0^{\sigma_k} V(X_s)\,ds} \bigr]
\\
&&\qquad = \sum_{p = k+2}^{n-2} \mathbf{E}^x
\bigl[S(n,p,l-1,t),S(p,k,1,t);
\\
&&\hspace*{78pt} e^{-(1/2)
\int_0^{\sigma_p} V(X_s)\,ds} e^{-(1/2) \int_{\sigma_p}^{\sigma_k} V(X_s)\,ds} \bigr]
\\
&&\qquad \leq \sum_{p = k+2}^{n-2}
\mathbf{E}^x \bigl[S(n,p,l-1,t),S(p,k,1,t+{\sigma_p});
\\
&&\hspace*{78pt}
e^{-(1/2) \int_0^{\sigma_p} V(X_s)\,ds} e^{-(1/2) \int_{\sigma_p}^{\sigma_k} V(X_s)\,ds} \bigr]
\\
&&\qquad = \sum_{p = k+2}^{n-2} \mathbf{E}^x
\bigl[S(n,p,l-1,t);
\\
&&\hspace*{78pt} e^{-(1/2)
\int
_0^{\sigma_p}V(X_s)\,ds} \mathbf{E}^{X_{\sigma_p}} \bigl[S(p,k,1,t);
e^{-(1/2) \int
_0^{\sigma_k}
V(X_s)\,ds} \bigr] \bigr].
\end{eqnarray*}
By the induction hypothesis and Lemma~\ref{lmlevykernel}(1), the last
sum is bounded above by
\[
\sum_{p = k+2}^{n-2} \frac{C_{31}}{2^{l-1} g_{n-2}} \int
_{R_{p}} \nu(y-x) \frac{C_{31}C_{28}}{2g_{p-2}} \int_{R_k}
\nu(z-y)\,dz\,dy.
\]
Hence, Fubini's theorem and Assumption~\ref{assassnu}(3) yield that
\begin{eqnarray*}
&& \mathbf{E}^x \bigl[S(n,k,l,t); e^{-(1/2) \int_0^{\sigma_k}
V(X_s)\,ds} \bigr]
\\
&&\qquad  \leq
\frac{C_{31}}{2^{l} g_{n-2}}\frac{C_{31}C_{28}}{g_{n_0}} \int_{R_k}\sum
_{p = k+2}^{n-2} \int_{R_{p}} \nu(y-x)
\nu(z-y) \,dy \,dz
\\
&&\qquad  \leq \frac{C_{31}}{2^{l} g_{n-2}}\frac{4 C_{30}C_{28} C_3}{g_{n_0}} \int_{R_k}
\nu(y-x) \,dy
\\
&&\qquad  \leq \frac{C_{31}}{2^{l} g_{n-2}} \int_{R_k} \nu(y-x) \,dy.
\end{eqnarray*}\upqed
\end{pf}

\subsection{Estimates of \texorpdfstring{$\lambda$}{$lambda$}-subaveraging functions}\label{sec3.3}
\mbox{}

\begin{pf*}{Proof of Theorem~\ref{thmdefic1}}
Recall that $C_{21}=C_{21}(X,V)$, $C_{22}=C_{22}(X,V)$ and
$C_{29}=C_{29}(X,m)$ are the constants in
(\ref{eqkhas}) and Lemma~\ref{lmlevykernel}(2), respectively. We write
\begin{eqnarray*}
C &=& C(X):= 2  \biggl(1 \vee\frac{C_2}{|B(0,1)| \nu((6,0,\ldots,0))}\biggr) \geq2,
\\
C^{(1)} &=& C^{(1)}(X):= \max_{1 \leq m \leq2} C_{29} \geq1,
\\
C^{(2)} &=& C^{(2)}(X):= \frac{1} 4  \biggl(1 \wedge\biggl(\int_{B(0,1)^c}\nu
(y)\,dy\biggr) ^{-1} \biggr) \leq\frac{1} 4.
\end{eqnarray*}

Notice\vspace*{1pt} that $C^{(2)} \int_{B(0,1)^c} \nu(y) \,dy \leq1/4$. For $n \in
\mathbf{N}
$ we denote $g_n:=\break\inf_{|y| \geq n}
V_0(y)$ with $V_0=V-\lambda$. Let $n_0$ be a natural number satisfying
the assumptions of Lemma
\ref{lmjumpest2} for the potential $V_0$ and the sequence
$(g_n)_{n\in\mathbb{N}}$,
and such that
%
\begin{eqnarray}
\label{eqset1}
&& \max \biggl(C_{31} \biggl( \biggl(\int
_{B(0,1)^c} \nu(y)\,dy \vee C_3 \biggr) +
\frac{3C_{21} C C_{28} C^{(1)}}{C^{(2)}} \biggr), 2(\lambda +C_{22}) \biggr)
\nonumber\\[-8pt]\\[-8pt]
&&\qquad \leq g_{n_0},\nonumber
\end{eqnarray}
where $C_{31}$ is the constant from Lemma~\ref{lmjumpest2}. It is
worth to note for later use that since $V_0(y) \leq2V(y)$ for $|y|
\geq n_0$, the number $n_0$ also satisfies the assumptions of Lemma
\ref
{lmjumpest2} for the potential $2 V$ and the same sequence
$(g_n)_{n\in\mathbb{N}}$.

We will show that
%
\begin{equation}
\label{eqvarphi0est1} \qquad\varphi(x) \leq C \llVert \varphi\rrVert _{\infty} \biggl(
\int_{R_{n_0}} \nu (y-x)\,dy \biggr)^{\sum_{i=1}^{p} 2^{-i}}\qquad\mbox{for }
|x| > n_0+3
\end{equation}
for all $p \in\mathbf{N}$. If this holds, then by taking the limit $p
\to
\infty$ and using Assumption~\ref{assassnu}(1) the theorem follows.

We use again induction. For more clarity, we divide the proof of (\ref
{eqvarphi0est1}) in two steps.
\begin{longlist}[\textit{Step} 2]
\item[\textit{Step} 1.] First, we show that the bound (\ref{eqvarphi0est1})
holds for $p=1$.
We have
\begin{eqnarray}
\varphi(x) \leq e^{\lambda t} \mathbf{E}^x \bigl[e^{-\int_0^t
V(X_s)\,ds} \varphi (X_t) \bigr] \leq \llVert \varphi\rrVert _{\infty}
\mathbf{E}^x \bigl[e^{-\int_0^t
V_0(X_s)\,ds} \bigr],\nonumber
\\
\eqntext{x \in \mathbf{R}^d, t>0.}
\end{eqnarray}
To get an upper bound on the latter expectation, let $n-1 < |x| \leq
n$, $n \geq n_0 + 4$. For all
$t >0$, we have
%
\begin{eqnarray}
\label{eqsubst1}
\qquad \mathbf{E}^x \bigl[e^{-\int_0^t V_0(X_s)\,ds} \bigr] & \leq&
\mathbf {E}^x \bigl[\tau_n > t; e^{-\int_0^t V_0(X_s)\,ds} \bigr]
\nonumber
\\[-8pt]
\\[-8pt]
&&{} + \sum_{k=n_0}^{n-2} \sum
_{l=1}^{\infty} \mathbf{E}^x
\bigl[S(n,k,l,t), \tau_k > t; e^{-\int_0^t V_0(X_s)\,ds} \bigr].
\nonumber
\end{eqnarray}
Clearly, by (\ref{eqset1}) the first term on the right-hand side is
estimated directly by
\[
\mathbf{P}^x (\tau_n > t) e^{-g_{n-2}t } \leq
\mathbf{P}^x (\tau_n > t) e^{-2(\lambda
+C_{22})t } \leq
e^{- (\lambda+C_{22})t }.
\]
Lemma~\ref{lmjumpest2} and (\ref{eqset1}) yield for $k \geq n_0+2$
\begin{eqnarray*}
&& \mathbf{E}^x \bigl[S(n,k,l,t), \tau_k > t;
e^{-\int_0^t
V_0(X_s)\,ds} \bigr]
\\
&&\qquad  \leq \mathbf{E}^x \bigl[S(n,k,l,t),
\tau_k > t; e^{-(1/2) \int
_0^{\sigma_k}V_0(X_s)\,ds} e^{-(1/2)\int_0^t V_0(X_s)\,ds} \bigr]
\\
&&\qquad  \leq e^{-(\lambda+C_{22})t} \mathbf{E}^x \bigl[S(n,k,l,t);
e^{-(1/2) \int
_0^{\sigma_k} V_0(X_s)\,ds} \bigr]
\\
&&\qquad \leq \frac{C_{31}}{g_{n_0}2^l} e^{- (\lambda+C_{22})t} \int_{R_k} \nu
(y-x) \,dy.
\end{eqnarray*}
Similarly, by the strong Markov property and (\ref{eqkhas}), we
have for $k=n_0$ and $k=n_0+1$ (recall $\tau_{n_0}=\tau_{n_0+1} =
\infty$)
\begin{eqnarray*}
&&  \mathbf{E}^x \bigl[S(n,k,l,t),
\tau_k > t; e^{-\int_0^t
V_0(X_s)\,ds} \bigr]
\\
&&\qquad \leq e^{\lambda t} \mathbf{E}^x \bigl[S(n,k,l,t);
e^{- \int_0^{\sigma_k}
V_{+}(X_s)\,ds}e^{\int_0^{t+\sigma_k} V_{-}(X_s)\,ds} \bigr]
\\
&&\qquad  =  e^{\lambda t} \mathbf{E}^x \bigl[S(n,k,l,t);
e^{- \int_0^{\sigma_k}
V(X_s)\,ds}e^{\int_{\sigma_k}^{t+\sigma_k} V_{-}(X_s)\,ds} \bigr]
\\
&&\qquad  =  e^{\lambda t} \mathbf{E}^x \bigl[S(n,k,l,t);
e^{- \int_0^{\sigma_k}
V(X_s)\,ds}\mathbf{E}^{X_{\sigma_k}} \bigl[e^{\int_0^t
V_{-}(X_s)\,ds} \bigr] \bigr]
\\
&&\qquad  =  e^{\lambda t} \sup_{y \in\mathbf{R}^d} \mathbf{E}^{y}
\bigl[e^{\int_0^t
V_{-}(X_s)\,ds} \bigr]\mathbf{E}^x \bigl[S(n,k,l,t);
e^{- \int
_0^{\sigma_k}
V(X_s)\,ds} \bigr]
\\
&&\qquad  \leq C_{21} e^{(\lambda+ C_{22})t} \mathbf{E}^x
\bigl[S(n,k,l,t); e^{-(1/2)
\int_0^{\sigma_k} 2V(X_s)\,ds} \bigr],
\end{eqnarray*}
which, in turn, by Lemmas~\ref{lmjumpest2}~and~\ref{lmlevykernel}(2), is smaller or equal to
\[
\frac{C_{21}C_{31}}{g_{n_0} 2^l} e^{(\lambda+ C_{22})t} \int_{R_k} \nu (y-x)\,dy
\leq \frac{C_{21}C_{31}C^{(1)}}{g_{n_0} 2^l} e^{(\lambda+ C_{22})t} \int_{R_{n_0}}
\nu(y-x)\,dy.
\]
By putting together the above estimates and choosing
\[
t= - \frac{1}{2(\lambda+ C_{22})} \log \biggl(C^{(2)}\int_{R_{n_0}}
\nu (y-x)\,dy \biggr) > 0
\]
in (\ref{eqsubst1}), and using (\ref{eqset1}) we conclude that
\begin{eqnarray*}
\varphi(x) & \leq&\llVert \varphi\rrVert _{\infty} \Biggl(
\biggl(C^{(2)}\int_{R_{n_0}}\nu(y-x)\,dy
\biggr)^{1/2}
\\
&&\hspace*{33pt}{}+ \frac{C_{31}}{g_{n_0}} \biggl(C^{(2)}\int_{R_{n_0}}
\nu(y-x)\,dy \biggr)^{1/2}\sum_{k=n_0+2}^{n-2}
\int_{R_k} \nu(y-x) \,dy
\\
&&\hspace*{114.5pt}{}+ \frac{2 C_{21}C_{31} C^{(1)}}{\sqrt{C^{(2)}}g_{n_0}} \biggl(\int_{R_{n_0}} \nu(y-x)\,dy
\biggr)^{1/2} \Biggr)
\\
& \leq& \llVert \varphi\rrVert _{\infty} \biggl(1 + \frac
{C_{31}}{g_{n_0}}
\biggl(\int_{B(0,1)^c} \nu(y)\,dy + \frac{2 C_{21} C^{(1)}}{\sqrt{C^{(2)}}} \biggr)
\biggr) \biggl(\int_{R_{n_0}}\nu(y-x)\,dy \biggr)^{1/2}
\\
& \leq& 2 \llVert \varphi\rrVert _{\infty} \biggl(\int_{R_{n_0}}
\nu (y-x)\,dy \biggr)^{1/2},
\end{eqnarray*}
which completes the first step.
\end{longlist}
\begin{longlist}[\textit{Step} 2.]
\item[\textit{Step} 2.] Suppose now that (\ref{eqvarphi0est1}) holds for $p$.
We prove that it is also
satisfied for $p+1$. We consider two cases. First let $n_0+3 < |x| \leq
n_0+4$. Denote $x_0 =
((n_0-1)/|x|)x$. Then by the fact that $1 \leq|y-x| \leq6$ for $y \in
B(x_0,1)$ and by Assumption
\ref{assassnu}(2), we have
\begin{eqnarray*}
\frac{\nu((6,0,\ldots,0)) |B(0,1)|}{C_2} &\leq&\int_{B(x_0,1)} \nu(y-x)\,dy
\\
&\leq&\int
_{R_{n_0}} \nu(y-x)\,dy.
\end{eqnarray*}
By using this estimate and the definition of $C$, it is immediate to
obtain the bound
\[
\varphi(x) \leq\llVert \varphi\rrVert _{\infty} \leq C \llVert \varphi
\rrVert _{\infty} \biggl(\int_{R_{n_0}}\nu(y-x)\,dy
\biggr)^{\sum_{i=1}^{p+1} 2^{-i}}.
\]
Let now $n-1 < |x| \leq n$, $n \geq n_0 + 5$. Similarly as before, for
all $t >0$ we have
%
\begin{eqnarray}
\label{eqsubst2} \varphi(x) &\leq& \mathbf{E}^x \bigl[
\tau_n > t; e^{-\int_0^t V_0(X_s)\,ds}\varphi(X_t) \bigr]
\nonumber
\\[-8pt]
\\[-8pt]
&&{} + \sum_{k=n_0}^{n-2} \sum
_{l=1}^{\infty} \mathbf{E}^x
\bigl[S(n,k,l,t), \tau _k > t; e^{-\int_0^t V_0(X_s)\,ds}\varphi(X_t)
\bigr].
\nonumber
\end{eqnarray}
By (\ref{eqvarphi0est1}), (\ref{eqset1}) and Lemma~\ref{lmlevykernel}(1), we find the following bound for the first
expectation in (\ref{eqsubst2})
%
\begin{eqnarray}\label{eqexpr1}
&& \mathbf{E}^x \bigl[\tau_n > t;
e^{-\int_0^t V_0(X_s)\,ds}\varphi (X_t) \bigr]\nonumber
\\
&&\qquad \leq e^{-2(\lambda+C_{22})t } \sup
_{|z| > n-2} \varphi(z)
\nonumber\\[-8pt]\\[-8pt]
&&\qquad  \leq e^{-(\lambda+C_{22})t} C \llVert \varphi\rrVert _{\infty} \sup
_{|z| > n-2} \biggl(\int_{R_{n_0}} \nu(y-z)\,dy
\biggr)^{\sum_{i=1}^{p} 2^{-i}}\nonumber
\\
&&\qquad \leq e^{-(\lambda+C_{22})t} C C_{28} \llVert \varphi\rrVert
_{\infty} \biggl(\int_{R_{n_0}} \nu(y-x)\,dy
\biggr)^{\sum_{i=1}^{p}
2^{-i}}.
\nonumber
\end{eqnarray}
We now estimate the expectations under the double sum on the right-hand
side of~(\ref{eqsubst2}).
By using (\ref{eqvarphi0est1}), (\ref{eqset1}) and the equality
$\sum_{i=1}^{p} 2^{-i} = 1-2^{-p}$,
we get for $k \geq n_0+3$
%
\begin{eqnarray}
\label{eqexpr2} && \mathbf{E}^x \bigl[S(n,k,l,t), \tau_k
> t; e^{-\int_0^t V_0(X_s)\,ds}\varphi (X_t) \bigr]
\nonumber
\\
&&\qquad\leq \sup_{|z| > k-2} \varphi(z) \mathbf{E}^x
\bigl[S(n,k,l,t), \tau_k > t;\nonumber
\\
&&\hspace*{101pt} e^{-(1/2) \int_0^{\sigma_k} V_0(X_s)\,ds}e^{-(1/2)\int_0^t
V_0(X_s)\,ds} \bigr]
\\
&&\qquad\leq C e^{-(\lambda+C_{22})t} \llVert \varphi\rrVert _{\infty} \sup
_{|z| > k-2} \biggl(\int_{R_{n_0}} \nu(y-z)\,dy
\biggr)^{1-2^{-p}}\nonumber
\\
&&\quad\qquad{}\times \mathbf{E}^x \bigl[S(n,k,l,t);
e^{-(1/2) \int_0^{\sigma_k} V_0(X_s)\,ds} \bigr].\nonumber
\end{eqnarray}
By Lemmas~\ref{lmlevykernel}(1) and~\ref{lmjumpest2}, the
right-hand side of (\ref{eqexpr2}) is
less or equal to
\[
\frac{C C_{31}}{g_{n_0}2^l} e^{-(\lambda+C_{22})t} \llVert \varphi \rrVert _{\infty}
\biggl(C_{28}\inf_{z \in R_k}\int_{R_{n_0}}
\nu(y-z)\,dy \biggr)^{1-2^{-p}} \int_{R_k} \nu(z-x) \,dz.
\]
Furthermore, by Fubini's theorem,
\begin{eqnarray*}
&& \inf_{z \in R_k} \int_{R_{n_0}} \nu(y-z)\,dy \int
_{R_k} \nu(z-x) \,dz
\\
&&\qquad \leq\int_{R_{n_0}}\int
_{R_k} \nu(y-z) \nu(z-x) \,dz \,dy
\end{eqnarray*}
and again by Lemma~\ref{lmlevykernel}(1),
\[
\biggl(C_{28}\inf_{z \in R_k}\int_{R_{n_0}}
\nu(y-z)\,dy \biggr)^{-2^{-p}} \leq \biggl(\int_{R_{n_0}}
\nu(y-x)\,dy \biggr)^{-2^{-p}}.
\]
Thus, the expectations on the left-hand side of (\ref{eqexpr2}) for $k
\geq n_0+3$ are bounded above by
%
\begin{eqnarray}\label{eqexpr3}
&& \frac{C C_{31}C_{28}}{g_{n_0}2^l} \llVert \varphi\rrVert _{\infty}
e^{- (\lambda+C_{22})t} \biggl(\int_{R_{n_0}} \nu(y-x)\,dy
\biggr)^{-2^{-p}}
\nonumber\\[-8pt]\\[-8pt]
&&\qquad{} \times\int_{R_{n_0}}\!\int_{R_k} \nu(y-z) \nu(z-x) \,dz \,dy.\nonumber
\end{eqnarray}
Similarly, by the strong Markov property, Lemmas~\ref{lmjumpest2} and
\ref{lmlevykernel}(2), we
estimate the expectations for $n_0 \leq k \leq n_0+2$ to obtain
%
\begin{eqnarray}
\label{eqexpr4} && \mathbf{E}^x \bigl[ S(n,k,l,t),
\tau_k > t; e^{-\int_0^t
V_0(X_s)\,ds}\varphi(X_t) \bigr]
\nonumber
\\
&&\qquad\leq \llVert \varphi\rrVert _{\infty}e^{\lambda t}
\mathbf{E}^x \bigl[S(n,k,l,t); e^{- \int_0^{\sigma_k} V_{+}(X_s)\,ds}e^{\int_0^{t+\sigma_k}
V_{-}(X_s)\,ds}
\bigr]
\nonumber
\\
&&\qquad = \llVert \varphi\rrVert _{\infty} e^{\lambda t}
\mathbf{E}^x \bigl[S(n,k,l,t); e^{- \int_0^{\sigma_k} V(X_s)\,ds}e^{\int_{\sigma_k}^{t+\sigma_k}
V_{-}(X_s)\,ds}
\bigr]
\nonumber
\\
&&\qquad= \llVert \varphi\rrVert _{\infty} e^{\lambda t}
\mathbf{E}^x \bigl[S(n,k,l,t); e^{- \int_0^{\sigma_k} V(X_s)\,ds}\mathbf{E}^{X_{\sigma_k}}
\bigl[e^{\int_0^t
V_{-}(X_s)\,ds} \bigr] \bigr]
\\
&&\qquad\leq C_{21} \llVert \varphi\rrVert _{\infty}
e^{(\lambda+ C_{22})t} \mathbf{E}^x \bigl[S(n,k,l,t);e^{-(1/2) \int_0^{\sigma_k}
2V(X_s)\,ds}
\bigr]
\nonumber
\\
&&\qquad\leq \frac{C_{21}C_{31}}{g_{n_0} 2^l} \llVert \varphi\rrVert _{\infty}
e^{(\lambda+ C_{22})t} \int_{R_k} \nu(y-x)\,dy
\nonumber
\\
&&\qquad\leq \frac{C_{21}C_{31}C^{(1)}}{g_{n_0} 2^l} \llVert \varphi\rrVert _{\infty}
e^{(\lambda+ C_{22})t} \int_{R_{n_0}} \nu(y-x)\,dy.
\nonumber
\end{eqnarray}
Combining the estimates (\ref{eqexpr1})--(\ref{eqexpr4}) and choosing
\[
t= -\frac{1}{\lambda+ C_{22}} \biggl(\log(C_{28} C)^{-1} +
2^{-(p+1)} \log \biggl(C^{(2)}\int_{R_{n_0}}
\nu(y-x)\,dy \biggr) \biggr)>0
\]
in (\ref{eqsubst2}), we obtain
\begin{eqnarray*}
\varphi(x) & \leq& \llVert \varphi\rrVert _{\infty} \Biggl( \biggl(\int
_{R_{n_0}} \nu (y-x)\,dy \biggr)^{\sum_{i=1}^{p+1} 2^{-i}}
\\
&&\hspace*{32pt}{} + \frac{C_{31}}{g_{n_0}} \biggl(\int_{R_{n_0}} \nu(z-x)\,dz
\biggr)^{2^{-(p+1)}-2^{-p}}
\\
&&\hspace*{42pt}{}\times \int_{R_{n_0}} \Biggl(\sum
_{k=n_0+3}^{n-2} \int_{R_k} \nu(y-z)\nu(y-x) \,dy \Biggr) \,dz
\\
&&\hspace*{32pt}{}+ \frac{3C_{21}C_{31}C C_{28} C^{(1)}}{C^{(2)}g_{n_0}} \biggl(\int_{R_{n_0}} \nu(y-x)\,dy
\biggr)^{\sum_{i=1}^{p+1} 2^{-i}} \Biggr).
\end{eqnarray*}
Finally, by using Assumption~\ref{assassnu}(3) and (\ref{eqset1}),
we get
\begin{eqnarray*}
\varphi(x) & \leq& \llVert \varphi\rrVert _{\infty} \biggl(1 +
\frac
{C_{31}C_3}{g_{n_0}} + \frac{3C_{21}C_{31}C
C_{28} C^{(1)}}{C^{(2)}g_{n_0}} \biggr) \biggl(\int_{R_{n_0}}
\nu (y-x)\,dy \biggr)^{\sum_{i=1}^{p+1} 2^{-i}}
\\
& \leq& 2 \llVert \varphi\rrVert _{\infty} \biggl(\int_{R_{n_0}}
\nu (y-x)\,dy \biggr)^{\sum_{i=1}^{p+1} 2^{-i}},
\end{eqnarray*}
which completes the proof of the theorem.\quad\qed
\end{longlist}\noqed
\end{pf*}

\begin{pf*}{Proof of Theorem~\ref{thmdefic2}}
By integrating in the equality $e^{-\lambda t}\varphi(x) = \mathbf
{E}^x
[e^{-\int_0^t V(X_s)\,ds}\varphi(X_t) ]$
over $t \in(0,\infty)$, it follows that
\[
\varphi(x) = \lambda G^V\varphi(x), \qquad x \in
\mathbf{R}^d,
\]
and by (\ref{eqpot1}) applied to $f=\varphi$, $D' = \mathbf{R}^d$,
$D=B(x,1)$, we obtain
%
\begin{equation}
\label{eqv0} \varphi(x) = \lambda G^V_D \varphi(x) +
\mathbf{E}^x \bigl[e^{-\int_0^{\tau_D}V(X_s)\,ds}\varphi(X_{\tau
_D})
\bigr], \qquad x \in\mathbf{R}^d.
\end{equation}

We now prove part (1) of the statement. Let $R>2$ be large enough so
that $V(x) \geq0$ for $|x| \geq R-1$ and
the assertion of Theorem~\ref{thmdefic1} for the $\lambda
$-subaveraging function $|\varphi|$ holds. Let $|x|
\geq R+2$. By (\ref{eqv0}), we have
\begin{eqnarray*}
\bigl|\varphi(x)\bigr| & \leq&\lambda G^V_D \bigl|\varphi(x)\bigr| +
\mathbf{E}^x \bigl[e_V(\tau_D)\bigl|
\varphi(X_{\tau_D})\bigr| \bigr] = {\mathrm{I}} + {\mathrm{II}}.
\end{eqnarray*}
By Theorem~\ref{thmdefic1} applied to $|\varphi|$, we have
\[
{\mathrm{I}} \leq C G^V_D \mathbf{1}(x) \sup
_{y \in D} \bigl|\varphi(y)\bigr| \leq C \llVert \varphi\rrVert
_{\infty} G^V_D \mathbf{1}(x) \nu(x)
\]
with $C=C(X,V,\lambda)$. To estimate ${\mathrm{II}}$ first note that by
Theorem~\ref{thmdefic1}, Assumption~\ref{assassnu}(1), (3), the Ikeda--Watanabe formula and the fact that
$\sup_{z \in D}
\mathbf{E}^z [\tau_D] \leq\mathbf{E}^0 [\tau_{B(0,2)}] < \infty$,
we have for $z \in
D \setminus B(x,3/4)$ the estimates
%
\begin{eqnarray}
\label{eqaux1}
&& \mathbf{E}^z \bigl[e_V(\tau_D)\bigl|\varphi(X_{\tau_D})\bigr| \bigr]\hspace*{-30pt}\nonumber
\\
&&\qquad  \leq  \mathbf{E}^z \bigl[\bigl|\varphi(X_{\tau_D})\bigr|; X_{\tau_D}
\in B(x,2)\setminus D \bigr] + \mathbf{E}^z \bigl[\bigl|
\varphi(X_{\tau _D})\bigr|; X_{\tau_D} \in B(x,2)^c \bigr]\hspace*{-30pt}\nonumber
\\
&&\qquad \leq C \biggl(\llVert \varphi\rrVert _{\infty} \nu(x) + \int
_D G_D(z,y) \int_{B(x,2)^c} \bigl| \varphi(w)\bigr|\nu(w-y) \,dw \,dy \biggr)\hspace*{-30pt}\nonumber
\\
&&\qquad \leq C \biggl(\llVert \varphi\rrVert _{\infty} \nu(x) +
\mathbf{E}^z [\tau_D] \int_{B(x,2)^c} \bigl|
\varphi(w)\bigr|\nu(w-x) \,dz \biggr)\hspace*{-30pt}
\\
&&\qquad  \leq C \llVert \varphi\rrVert _{\infty} \biggl(\nu(x) + \int
_{B(0,R)^c\cap B(x,2)^c} \nu(w) \nu(w-x) \,dz\hspace*{-30pt}\nonumber
\\
&&\hspace*{151pt}{} + \int_{B(0,R)} \nu(w-x) \,dw \biggr)\hspace*{-30pt}\nonumber
\\
&&\qquad  \leq C \llVert \varphi\rrVert _{\infty} \nu(x)\hspace*{-30pt}\nonumber
\end{eqnarray}
with $C=C(X,V,\lambda)$. Thus, by using Corollary~\ref{lmbhi}, the
above estimate, Theorem~\ref{thmdefic1}
and Assumption~\ref{assassnu}(1), (3), we finally have
\begin{eqnarray*}
{\mathrm{II}} & \leq& C G^V_D\mathbf{1}(x) \biggl(\int
_{D \cap B(x,3/4)^c} \mathbf{E}^z \bigl[e_V(\tau
_D)\bigl|\varphi(X_{\tau_D})\bigr| \bigr]\nu(z-x) \,dz
\\
&&\hspace*{143pt} {}+ \int_{D^c} \bigl|\varphi(z)\bigr|\nu(z-x) \,dz \biggr)
\\
& \leq& C \llVert \varphi\rrVert _{\infty} G^V_D
\mathbf{1}(x) \biggl( \nu (x) \int_{D
\cap B(x,3/4)^c} \nu(z-x) \,dz
\\
&& \hspace*{76pt} {}+ \int_{D^c \cap B(0,R)^c} \nu(z) \nu(z-x) \,dz
+ \int_{B(0,R)}\nu(z-x) \,dz\biggr)
\\
& \leq& C \llVert \varphi\rrVert _{\infty} G^V_D
\mathbf{1}(x) \nu(x),
\end{eqnarray*}
where $C=C(X,V,\lambda)$. We conclude that $|\varphi(x)| \leq C_5
\llVert \varphi\rrVert _{\infty} G^V_D \mathbf{1}(x) \nu(x)$ for all $|x|
\geq R+2$, with constant $C_5=C_5(X,V,\lambda)$.

Now consider part (2) of the statement. Again, by (\ref{eqv0}), strict
positivity of $\varphi$ and Corollary~\ref{lmbhi} we have
\begin{eqnarray}
\varphi(x) \geq\mathbf{E}^x \bigl[e^{-\int_0^{\tau
_D}V(X_s)\,ds} \varphi(X_{\tau
_D}) \bigr] \geq C^{-1}_{24}
G^V_{B(x,1)}\mathbf{1}(x) \int_{B(0,1)}
\varphi(z)\nu (x-z) \,dz,\nonumber
\\
\eqntext{|x| \geq R.}
\end{eqnarray}
By Assumption~\ref{assassnu}(1), the last integral is greater than $C
\nu(x) \int_{B(0,1)} \varphi(z) \,dz$
and the required inequality follows again from the positivity of
$\varphi$ with constant $C_6=C_6(X,\varphi)$.
\end{pf*}

\subsection{Eigenfunction estimates}\label{sec3.4}
\mbox{}

\begin{pf*}{Proof of Theorem~\ref{thmbsest}}
Let $\eta\geq0$ be such that $\lambda_0+\eta>0$ and let $n \geq0$
be fixed. We thus clearly have
$\varphi_n(x) = e^{\lambda t} \mathbf{E}^x[e_{V+\eta}(t)\varphi
_n(X_t)]$, $x
\in\mathbf{R}^d$, with $\lambda=
\lambda_n+\eta> \lambda_0+\eta>0$, and the result immediately follows
from Theorem~\ref{thmdefic2}(1)
for $\varphi= \varphi_n$.
\end{pf*}

\begin{pf*}{Proof of Theorem~\ref{thmgsest}}
Let $\eta\geq0$ be such that $\lambda_0+\eta>0$. The result directly
follows from Theorems~\ref{thmbsest}
and~\ref{thmdefic2}(2) for $\varphi= \varphi_0 > 0$ and $\lambda=
\lambda_0 + \eta> 0$.
\end{pf*}

\subsection{Intrinsic ultracontractivity}\label{sec3.5}
\mbox{}

\begin{pf*}{Proof of Theorem~\ref{thmiucvelgsd}}
First,\vspace*{1pt} we prove (1). By a standard argument based on the duality and
symmetry of $\widetilde T_t$, we have
$\|\widetilde T_t\|_{1 \rightarrow2} = \|\widetilde T_t \|_{2
\rightarrow\infty}$. Since
$\|\widetilde T_{t_0}\|_{2 \rightarrow\infty} < \infty$, by the
semigroup property it is straightforward
that also $\|\widetilde T_{2t_0}\|_{1 \rightarrow\infty} \leq\|
\widetilde T_{t_0}\|^2_{2 \rightarrow\infty}$.
Hence,
\[
\frac{e^{2 \lambda_0 t_0}}{\varphi_0(x)} T_{2t_0}\mathbf {1}(x)=\widetilde T_{2t_0}
\biggl(\frac{1}{\varphi_0} \biggr) (x) \leq C_{t_0},
\]
since by Theorem~\ref{thmbsest} we have $1/\varphi_0 \in L^1(\mathbf{R}
^d,\varphi_0^2\,dx)$. Hence, the implications
\mbox{$t_0$-}IUC${}\Rightarrow 2t_0$-GSD and IUC${}\Rightarrow{}$GSD follow. To
show that  $t_0$-GSD${}\Rightarrow 2t_0$-IUC ($t_0 \geq t_b$),
it suffices to observe that for all $f \in L^2(\mathbf{R}^d, \varphi
_0^2\,dx)$ we have
\begin{eqnarray*}
T_{2t_0}(f\varphi_0) (x) &=& T_{t_0}
T_{t_0}(f\varphi_0) (x)
\\
&\leq& T_{t_0} \mathbf{1}
(x) \|T_{t_0}\|_{2 \rightarrow\infty} \|f\varphi_0
\|_2
\\
&\leq& C_{t_0} \|f\varphi_0\|_2
\varphi_0(x).
\end{eqnarray*}
The last inequality follows from $\|T_{t_0}\|_{2 \rightarrow\infty
}<\infty$, coming from the boundedness of
$p(t_0,x)$ in $x \in\mathbf{R}^d$. Moreover, since GSD means $t$-GSD
for all
$t>0$, assertion~(2) of the theorem
follows again by the latter estimate.
\end{pf*}

\begin{lemma}
\label{lmequiv}
Let $V$ be a $X$-Kato class potential, nonnegative outside a bounded
subset of $\mathbf{R}^d$. Let Assumptions
\ref{assassnu}--\ref{assassbhi} be satisfied and consider the
following two conditions:
\begin{longlist}[(2)]
\item[(1)]
There exist $C_{32}=C_{32}(X,V,t)$ and $R >0$ such that
%
\begin{equation}
\label{eqcruest1} T_t \mathbf{1}(x) \leq C_{32} \nu(x),
\qquad|x| \geq R.
\end{equation}
\item[(2)]
There exist $C_{33}=C_{33}(X,V,t)$ and $R >0$ such that
%
\begin{equation}
\label{eqcruest2} T_t \mathbf{1}(x) \leq C_{33}
G^V_{B(x,1)}\mathbf{1}(x) \nu (x), \qquad|x| \geq R.
\end{equation}
\end{longlist}
Statements (1) and (2) are equivalent in the following sense. If (2) is
true for some $t=s>0$, then (1)
also follows for $t=s$. If (1) holds true for some $t=s>0$, then~(2)
follows for $t=2s$.
\end{lemma}

\begin{pf}
For the proof of the implication (2)${}\Rightarrow{}$(1), it suffices to
note that there is $R>0$ large
enough such that $G_{B(x,1)}^V\mathbf{1}(x) \leq G_{B(x,1)}^0 \mathbf
{1}(x) = \mathbf{E}^x
[\tau_{B(x,1)}] = \mathbf{E}^0 [\tau_{B(0,1)}]
< \infty$ for $|x| \geq R$.

For the converse implication, we suppose that (1) holds for some
$t/2>0$ and find $R_1 \geq R \vee2$ such
that $V(x) \geq0$ for $|x| \geq R_1-1$. Denote $D=B(x,1)$ and let $|x|
\geq R_1+2$. By Lemma~\ref{lmtechiu}
and Corollary~\ref{lmbhi}, we have
\begin{eqnarray*}
T_t\mathbf{1}(x) & =& \mathbf{E}^x \biggl[
\frac{t}{2}<\tau_D; e_V(t) \biggr] + \mathbf
{E}^x \biggl[\frac{t}{2} \geq\tau_D;
e_V(t) \biggr]
\\
& \leq& C \Bigl(G_D^V\mathbf{1}(x) \sup
_{y \in D}T_{t/2}\mathbf {1}(y) \Bigr)
 +
\mathbf{E}^x \bigl[e_V(\tau_D)T_{t/2}
\mathbf{1}(X_{\tau_D}) \bigr]
\\
& \leq& C G_D^V\mathbf{1}(x) \biggl(\sup
_{y \in B(x,1)}T_{t/2}\mathbf {1}(y)
\\
&&\hspace*{48.5pt}{}+ \int
_{B(x,1) \cap B(x,3/4)^c} \mathbf{E}^z \bigl[e_V(\tau
_D)T_{t/2}\mathbf{1}(X_{\tau
_D}) \bigr] \nu(z-x) \,dz
\\
&&\hspace*{48.5pt}{}+ \int_{B(x,1)^c \cap B(0,R_1)^c} T_{t/2}\mathbf{1}(z)\nu(z-x)\,dz
\\
&&\hspace*{107pt}{}+ \sup_{y \in B(0,R_1)}T_{t/2}\mathbf{1}(y) \int
_{B(0,R_1)} \nu(z-x) \,dz\biggr).
\end{eqnarray*}
Notice that by using (1) and exactly the same arguments as in (\ref
{eqaux1}) applied to $|\varphi(\cdot)|$
replaced by $T_{t/2}\mathbf{1}(\cdot)$, we get
%
\begin{equation}
\label{eqaux2} \mathbf{E}^z \bigl[e_V(
\tau_D)T_{t/2}\mathbf{1}(X_{\tau_D}) \bigr] \leq C
\nu(x), \qquad z \in D \cap B(x,3/4)^c
\end{equation}
with constant $C=C(X,V,t)$. Thus, by estimate (1), Assumption~\ref{assassnu}(1), (3) and~(\ref{eqaux2}),
we conclude similarly as in the proof of Theorem~\ref{thmdefic2} that
\[
T_t \mathbf{1}(x) \leq C G^V_{B(x,1)}
\mathbf{1}(x) \nu(x), \qquad|x| \geq R_1+2
\]
with $C=C(X,V,t)$, which completes the proof.
\end{pf}

\begin{theorem}
\label{thmcruiuc}
Let Assumptions~\ref{assassnu}--\ref{assasspinning} hold. If there
exist a constant $C_{13}$ and
$R > 0$ such that
%
\begin{equation}
\label{assassV1} \frac{V(x)}{|\log\nu(x)|} \geq C_{13}, \qquad|x| \geq R,
\end{equation}
then the bound (\ref{eqcruest1}) holds for all $t \geq t_0 =
2/C_{13}$. If, moreover,
%
\begin{equation}
\label{assassV2} \lim_{|x| \to\infty} \frac{V(x)}{|\log\nu(x)|} = \infty,
\end{equation}
then this bound holds for every $t>0$.
\end{theorem}

\begin{pf}
First assume that the inequality (\ref{assassV1}) is satisfied for $R>0$,
and denote $t_0 = 2/C_{13}$. Choose $n_0 \geq R$ large enough such that
\[
C_2 \nu\bigl((n,0,\ldots,0)\bigr) < 1\quad\mbox{and}\quad 2
\theta_0 \leq g_n:= \inf_{|y|\geq n}V(y)
\qquad\mbox{for } n \geq n_0
\]
and
\[
C_3 C_{28} C_{31}\leq g_{n_0},
\]
where $C_{31}$ is the constant and $\theta_0$ is the parameter from
Lemma~\ref{lmjumpest2}. Thus, the
assumptions of Lemma~\ref{lmjumpest2} are satisfied. Moreover, by
(\ref{assassV1}) and Assumption
\ref{assassnu}(2),
%
\begin{equation}
\label{eqgnvelnu} g_{n} \geq- C_{13} \log
\bigl(C_2 \nu\bigl((n,0,\ldots,0)\bigr)\bigr) \qquad \mbox {for } n
\geq n_0.
\end{equation}

We show that for every $t \geq t_0$ condition (\ref{eqcruest1}) holds.
Let $n-1 < |x| \leq n$,
$n \geq n_0 + 4$ and $t \geq t_0$. We have
%
\begin{eqnarray}\label{eqttest1}
T_t \mathbf{1}(x) &\leq&\mathbf{E}^x \bigl[\tau_n > t; e^{-\int_0^t
V(X_s)\,ds} \bigr]
\nonumber\\[-8pt]\\[-8pt]
&&{} + \sum _{k=n_0}^{n-2} \sum_{l=1}^{\infty}
\mathbf{E}^x \bigl[S(n,k,l,t), \tau_k > t;
e^{-\int_0^t V(X_s)\,ds} \bigr].\nonumber
\end{eqnarray}
By (\ref{eqgnvelnu}) and Assumption~\ref{assassnu}(1), the first
term at the right-hand side can
be easily estimated by
\[
\mathbf{P}^x (\tau_n > t) e^{- t g_{n-2}} \leq
e^{C_{13} t \log(C_2
\nu
((n-2,0,\ldots,0)))} \leq C^2_1 C_2 \nu(x).
\]
Similar arguments and Lemma~\ref{lmjumpest2} also yield
\begin{eqnarray*}
&& \mathbf{E}^x \bigl[S(n,k,l,t), \tau_k > t;
e^{-\int_0^t V(X_s)\,ds} \bigr]
\\
&&\qquad  \leq \mathbf{E}^x \bigl[S(n,k,l,t),
\tau_k > t; e^{-(1/2) \int_0^{\sigma_k} V(X_s)\,ds}e^{-(1/2)\int_0^t
V(X_s)\,ds} \bigr]
\\
&&\qquad \leq e^{(1/2)C_{13} t \log(C_2 \nu((k-2,0,\ldots,0)))} \mathbf{E}^x \bigl[S(n,k,l,t);
e^{-(1/2) \int_0^{\sigma_k}
V(X_s)\,ds} \bigr]
\\
&&\qquad  \leq \frac{C^2_1 C_2 C_{31}}{2^lg_{n_0}} \int_{R_k}\nu(y) \nu(y-x)\,dy
\end{eqnarray*}
for $k \geq n_0+2$. For $k \in \{n_0, n_0+1 \}$ we have
\begin{eqnarray*}
&& \mathbf{E}^x \bigl[S(n,k,l,t), \tau_k > t;
e^{-\int_0^t
V(X_s)\,ds} \bigr]
\\
&&\qquad  \leq \mathbf{E}^x \bigl[S(n,k,l,t);
e^{- \int_0^{\sigma_k}
V_{+}(X_s)\,ds}e^{\int
_0^{t+\sigma_k} V(X_s)\,ds} \bigr]
\\
&&\qquad  = \mathbf{E}^x \bigl[S(n,k,l,t); e^{- \int_0^{\sigma_k}
V(X_s)\,ds}e^{\int_{\sigma
_k}^{t+\sigma_k} V_{-}(X_s)\,ds}
\bigr]
\\
&&\qquad = \mathbf{E}^x \bigl[S(n,k,l,t); e^{- \int_0^{\sigma_k} V(X_s)\,ds}
\mathbf{E}^{X_{\sigma_k}} \bigl[e^{\int_0^t
V_{-}(X_s)\,ds} \bigr] \bigr]
\\
&&\qquad \leq C_{21}e^{C_{22}t} \mathbf{E}^x
\bigl[S(n,k,l,t); e^{-(1/2)
\int_0^{\sigma_k} 2V(X_s)\,ds} \bigr]
\\
&&\qquad  \leq \frac{C_{21} C_{31} e^{C_{22}t}}{2^l g_{n_0}}\int_{R_k} \nu(y-x)\,dy
\end{eqnarray*}
by the strong Markov property, (\ref{eqkhas}) and Lemma~\ref{lmjumpest2}. Thus, by Assumption~\ref{assassnu}(1) and (3), the second term at the right-hand side of (\ref
{eqttest1}) is bounded above by
\begin{eqnarray*}
&& \sum_{k=n_0}^{n-2} \sum
_{l=1}^{\infty}  \mathbf{E}^x
\bigl[S(n,k,l,t), \tau_k > t; e^{-\int_0^t
V(X_s)\,ds} \bigr]
\\
&&\qquad \leq C \sum_{l=1}^{\infty} 2^{-l}
\Biggl(\sum_{k=n_0}^{n_0+1} \int
_{R_k} \nu (y-x)\,dy + \sum_{k=n_0+2}^{n-2}
\int_{R_k} \nu(y) \nu(y-x)\,dy \Biggr)
\\
&&\qquad \leq C \nu(x),
\end{eqnarray*}
where $C=C(X,V,t)$, and the first part of the theorem is proved. The
second assertion follows from the first
part by observing that (\ref{assassV2}) implies (\ref{assassV1}) with
arbitrarily large constant $C_{13}$.
\end{pf}

\begin{pf*}{Proof of Theorem~\ref{thmsuffiuc}}
We first prove (1). By Theorems~\ref{thmcruiuc}~and~\ref{thmgsest}, Lemma~\ref{lmequiv},
there is $R>0$ such that for all $|x|>R$ condition (\ref{eqdefiuc})
holds with $t_0 = 4/C_{13}$. The
same is true for $|x|\leq R$ by boundedness of $T_t\mathbf{1}$, and by
continuity and strict positivity of
$\varphi_0$.

To prove (2) fix $\varepsilon\in(0,1]$ and first note that for every
$t>0$, we have that 
$\mathbf{P}^0(t<\tau_{B(0,\varepsilon)})>0$. This positivity
property follows
from \cite{bibPru} for small
$t>0$ and extends to all $t>0$. By definition of $t_0$-GSD and Theorem
\ref{thmgsest}, there is
$R > 0$ such that
\[
e^{- t_0 \sup_{|y-x|<\varepsilon} V(y)} \mathbf{P}^x(t_0< \tau
_{B(x,\varepsilon
)}) \leq T_{t_0} \mathbf{1}(x) \leq C
\varphi_0(x) \leq C \nu(x)
\]
with $C=C(X,V,t_0)$, for all $|x| > R$. Thus, by the fact that $|\log
\nu(x)| \to\infty$ as $|x| \to\infty$,
we can choose $R_{\varepsilon} \geq R$ large enough such that
\[
\frac{\sup_{|y-x|<\varepsilon}V(y)}{|\log\nu(x)|} \geq \frac{1}{t_0} \biggl(1 - \frac{\log(
C/({\mathbf{P}^0(t_0< \tau_{B(0,\varepsilon)})})}{|\log\nu(x)|} \biggr)
\geq\frac
{1}{2t_0}
\]
for $|x| > R_{\varepsilon}$, which gives the required bound.
\end{pf*}

\begin{pf*}{Proof of Theorem~\ref{thmneciuc}}
To prove (1), observe that for any $t>0$ we can find $R>0$ such that
$V(x) \geq(4/t) |\log\nu(x)|$
for $|x|>R$, and we can proceed in the same way as in the proof of (1)
of Theorem~\ref{thmsuffiuc}.

When the semigroup $\{T_t\dvtx  t \geq0\}$ is IUC, then by the same
arguments as in
the proof of Theorem
\ref{thmsuffiuc}(2) for every $t>0$ there is $C=C(X,V,t)$ such that
we have
\[
\frac{\sup_{|y-x|<\varepsilon}V(y)}{|\log\nu(x)|} \geq \frac{1}{t} \biggl(1 - \frac{\log(C/({\mathbf{P}^0(t< \tau
_{B(0,\varepsilon
)})})}{|\log\nu(x)|}
\biggr), \qquad|x|>R, t>0
\]
for some $R >0$. Thus, $\liminf_{|x| \to\infty} \frac{\sup_{|y-x|<\varepsilon}V(y)}{|\log\nu(x)|} =
\frac{1}{t}$, $t >0$, which completes the proof.
\end{pf*}

\begin{pf*}{Proof of Proposition~\ref{propultracontractivity}}
Denote $D = B(x_0,\varepsilon)$ and let $M =\break \sup_{y \in D} V(y)$. By
assumption (1) and Lemma~\ref{lmcompden},
$\lim_{y \to x} p_D(t,x,y) = \infty$ for every $x \in D$. Using this,
we can derive that the transition operator
of the process $(X_t)_{t \geq0}$ killed in $D$, that is, $P^D_t f(x) =
\int_D
p_D(t,x,y) f(y) \,dy$, $f \in L^1(D)$, is not
bounded from $L^1(D)$ to $L^{\infty}(D)$. Since for $f \geq0$, we have
\begin{eqnarray}
T_t f(x) \geq\mathbf{E}^x \bigl[e^{-\int_0^t V(X_s)\,ds}
f(X_t); t < \tau _D \bigr] \geq e^{-Mt} \int
_{D} p_D(t,x,y) f(y) \,dy,\nonumber
\\[-3pt]
\eqntext{x \in D,}
\end{eqnarray}
this clearly means that $T_t$ is not a bounded operator from
$L^1(\mathbf{R}
^d)$ to $L^{\infty}(\mathbf{R}^d)$ as well. Thus,
also $T_{t/2}$ cannot be a bounded operator from $L^2(\mathbf{R}^d)$ to
$L^{\infty}(\mathbf{R}^d)$.
\end{pf*}

\begin{pf*}{Proof of Theorem~\ref{thmcharacter}}
First, assume that the potential $V$ is such that the semigroup $\{T_t\dvtx
t \geq0\}$ is $t_0$-GSD. From Corollary
\ref{corborder}(2), we directly derive that there is $R>0$ such that
$2C_{14}t_0V(x)\nu(x) \geq
- \nu(x) \log\nu(x)$, for all $x \in B(0,R)^c$. By integrating in this
inequality with respect to
Lebesgue measure over an arbitrary Borel set $A \subset B(0,R)^c$, we
obtain $F_A^{2C_{14}t_0V}(\nu)
\geq0$.

Consider the converse implication. Since for some $t_0>0$, $R>0$ and
any Borel set $A \subset B(0,R)^c$,
we have $\int_A \nu(x)  (|\log\nu(x)| - t_0V(x) )\,dx \leq0$,
the bound $|\log\nu(x)| \leq
t_0V(x)$ holds for almost every $x \in B(0,R)^c$. By Assumptions~\ref{assassnu}(1) and~\ref{assasscomp},
it is immediate to deduce that there is $R_1>1$ such that $|\log\nu
(x)| \leq2 C_{14} t_0 V(x)$ for all
$|x|>R_1$. Again, by Corollary~\ref{corborder}(2), this implies $8
C_{14} t_0$-GSD of $\{T_t\dvtx  t \geq0\}$.
\end{pf*}

\section{Discussion of examples}\label{sec4}

\subsection{Verification of assumptions for the class of L\'evy processes considered}\label{sec4.1}
\label{subsecass}
In the first example below, we show various choices of structure of the
L\'evy measure~$\nu$ that satisfy
conditions (1)--(3) in Assumption~\ref{assassnu}.

\begin{example}\label{exex1}
(1)~Choosing $\nu(x) \asymp|x|^{-d-\alpha}(1+|x|)^{\alpha-\beta}$, $x
\in \mathbf{R}^d$, for $\alpha\in[0,2)$ and
$\beta> 0$, it can be directly seen that the conditions are verified.

(2)~Also, if $\nu(x) \asymp\kappa(|x|) |x|^{-d-\alpha}$, $\alpha\in
(0,2)$, where $\kappa\dvtx
[0,\infty) \to(0, 1]$ is a nonincreasing function such that $\kappa
(0)=1$ and $\kappa(a)\kappa(b) \leq
C \kappa(a + b)$, $a, b, C>0$, then all conditions on $\nu$ are
verified directly. Examples include
$\kappa(s)= 1/\log(e+s)$ and $\kappa(s)= 1/(\log(e + \log(1+s)))$.

(3)~A case of special interest is $\nu(x) \asymp e^{-a|x|^{\beta}}
|x|^{-d-\delta} (1+|x|)^{d+\delta-\gamma}$ with
\mbox{$a > 0$}, $\beta> 0$, $\delta\in[0,2)$ and $\gamma> 0$. In this case
condition~(2) always holds, condition (1) is satisfied when $\beta\in(0,1]$
without further restriction, and condition (3)
is satisfied when moreover $\gamma> (d+1)/2$.
\end{example}

We also give counterexamples to condition (3) in Assumption~\ref{assassnu}.

\begin{example} 
\label{exneg1}
For $\beta>1$ in case (3) of Example~\ref{exex1} the condition (3) of
Assumption~\ref{assassnu} is not
satisfied. Similarly, at least in one dimension, it fails when $\beta
=1$ and $\gamma= (d+1)/2$.
\end{example}

In the group of Examples~\ref{exex2}--\ref{exex5} next we discuss
specific classes of L\'evy processes
satisfying all of Assumptions~\ref{assassnu}--\ref{assassbhi}.

\begin{example}
\label{exex2}
\emph{Subordinate Brownian motions} with characteristic
exponents $\psi$ such that $e^{-t_{b} \psi(\cdot)} \in L^1(\mathbf
{R}^d)$ for
some $t_{b}>0$, whenever their L\'evy
measures satisfy Assumption~\ref{assassnu}. Since in this case $\nu
(x)$ is radially decreasing, condition
(2) of Assumption~\ref{assassnu} is automatically satisfied, however,
not necessarily the remaining
conditions (1) and (3). Condition (1) is always satisfied as long as
$\nu(x) \leq C \nu(y)$ for all $|x|
\geq1$, $|y|=|x|+1$, while, as seen in Example~\ref{exex1}, condition
(3) strongly depends on the
specific form of the L\'evy measure (in fact, the L\'evy measure of the
subordinator).
The transition densities $p(t,y-x)$ are given by the subordination
formula, that is, by the integral over time of
the Brownian transition kernel with respect to the distribution of the
given subordinator. Since also
$e^{-t_{b} \psi(\cdot)}$ is integrable, Assumption~\ref{assassdensity}
is satisfied [in particular, $(X_t)_{t \geq0}$
is not a compound Poisson process]. Lastly, Assumption~\ref{assassbhi}
follows from a similar bound for the
potential or \mbox{$\lambda$-}potential kernel, which is again a consequence
of the subordination formula and an easy
estimate. Below we give specific examples of subordinate Brownian
motion of special interest satisfying
all of our assumptions. For properties of subordinate Brownian motion
and further examples, see
\cite{bibBBKRSV,bibKSV,bibSSV,bibBer2}.
\begin{longlist}[(2)]
\item[(1)]
\emph{Rotationally symmetric $\alpha$-stable process}.
Let $\psi(\xi)=|\xi|^{\alpha}$, $\alpha\in(0,2)$. In this case,
$\nu
(x) = C(\alpha) |x|^{-d-\alpha}$.

\item[(2)]
\emph{Mixture of independent rotationally symmetric stable processes
with indices $\alpha$ and $\beta$}.
This is obtained for $\psi(\xi)=a |\xi|^{\alpha}+ b |\xi|^{\beta}$,
$0<\beta<\alpha<2$, $a, b >0$. We
have $\nu(x) = aC(\alpha)|x|^{-d-\alpha}+bC(\beta)|x|^{-d-\beta}$.

\item[(3)]
\emph{Jump-diffusion process} \cite{bibCKSo3,bibCKu2}.
Let $\psi(\xi)= a |\xi|^{\alpha}+ b |\xi|^2$, $0<\alpha<2$, $a, b >0$,
that is, the process is a mixture of
a rotationally symmetric $\alpha$-stable process and an independent
Brownian motion. In this case $\nu(x)
\asymp a C(\alpha) |x|^{-d-\alpha}$.

\item[(4)]
\emph{Rotationally symmetric geometric $\alpha$-stable process} \cite{bibSiSV,bibGrzR}.
Let $\psi(\xi)= \log(1+|\xi|^{\alpha})$, $0<\alpha<2$. In this case,
$\nu(x) \asymp|x|^{-d}(1+|x|)^{-\alpha}$.
Notice that $\psi$ is now a slowly varying function at infinity. In
contrast to the previous examples,
in this case there is $t(\alpha)>0$ for which the transition
probability densities are unbounded for
$0<t<t(\alpha)$ (\cite{bibBBKRSV}, page~117), though they are bounded for
large $t$.

\item[(5)]
\emph{Relativistic rotationally symmetric $\alpha$-stable process}
\cite{bibR,bibCKSo2}.
Let\break  $\psi(\xi)=(|\xi|^2 + m^{2/\alpha})^{\alpha/2}-m$, $\alpha\in
(0,2)$, $m>0$. It is known that $\nu(x)
\asymp\break e^{-m^{1/\alpha}|x|}|x|^{-d-\alpha}(1+|x|^{(d+\alpha-1)/2})$
\cite{bibKS} [we take $a=m^{1/\alpha}$,
$\beta= 1$, $\delta=\alpha$, $\gamma= (d+\alpha+1)/2$ in (3) of
Example~\ref{exex1}].
\end{longlist}
\end{example}

\begin{example}
\label{exex3}
\emph{Symmetric L\'evy processes with nondegenerate Brownian part}~\cite{bibKSV2}.
Let $(X_t)_{t \geq0}$ be a L\'evy process with characteristic exponent
$\psi(\xi
)=c|\xi|^2+
\int_{\mathbf{R}^d}(1-\cos(z \cdot\xi))\nu(dz)$, $c>0$, that is,
a sum of
Brownian motion with rescaled time
$(B_{2ct})_{t \geq0}$ and an independent symmetric L\'evy process
$(Y_t)_{t \geq0}$ with L\'evy measure $\nu$
satisfying Assumption~\ref{assassnu}. In this case, the transition
densities are given by the convolution
of the Gaussian kernel and the distribution of the process $(Y_t)_{t
\geq0}$
[note that we do not need to assume
that $(Y_t)_{t \geq0}$ has transition densities] or by the Fourier inversion
formula. They are clearly bounded
for all $t>0$, thus Assumption~\ref{assassdensity} also is satisfied.
In one dimension, Assumption
\ref{assassbhi} easily follows from a similar bound for the
corresponding $\lambda$-potential kernel.
In higher dimensions, the required upper bound for the $\lambda
$-potential kernel can be proved by
showing that, for instance, for every $\varepsilon>0$ there is
$t_{\varepsilon}>0$ such that
%
\begin{equation}
\label{eqregcond} \mathbf{P}^0(Y_t \in A) \leq C |A| \qquad
\mbox{whenever } A \subset\mathcal{B}\bigl(\mathbf{R} ^d\bigr),
\operatorname{dist}(0,A)>\varepsilon, t \in(0,t_{\varepsilon})\hspace*{-30pt}
\end{equation}
(here $|A|$ denotes Lebesgue measure of $A$) with a constant
$C=C(Y,\varepsilon)$ independent of $t$
and the specific $A$. Specific cases are jump-diffusions as above, and
many similar processes in which
the rotationally symmetric stable process is replaced by other
symmetric L\'evy processes $(Y_t)_{t \geq0}$
satisfying Assumption~\ref{assassnu} and condition~(\ref{eqregcond}).
\end{example}

\begin{example}
\label{exex4}
\emph{Symmetric stable-like L\'evy processes} \cite{bibCKu}.
Let $\alpha\in(0,2)$ and $(X_t)_{t \geq0}$ be a purely jump (i.e.,
with no
diffusion part) symmetric L\'evy
process with intensity $\nu(x) \asymp C |x|^{-d-\alpha}$, $x \in
\mathbf{R}^d$.
It is known \cite{bibCKu} that
$(X_t)_{t \geq0}$ has bounded continuous transition probability
densities $p(t,y-x) \asymp t^{-d/\alpha}
\wedge t |y-x|^{-d-\alpha}$, $t>0$, $x, y \in\mathbf{R}^d$. In this case
Assumptions~\ref{assassnu} and
\ref{assassdensity} are clearly satisfied, while Assumption~\ref{assassbhi} is an easy consequence
of a similar bound for the potential kernel ($\alpha<d$) or the
$\lambda
$-potential kernel ($1=d \leq
\alpha< 2$) of $(X_t)_{t \geq0}$. This class includes a subclass of strictly
stable L\'evy processes with
intensities of the form $|x|^{-d-\alpha} f(x/|x|)$ with functions $f(x)
= f(-x)$ that are bounded from
above and below by positive constants \cite{bibBSz}.
\end{example}

\begin{example}
\label{exex5}
\emph{Symmetric L\'evy processes with subexponentially localized
L\'evy measures}.
Let $(X_t)_{t \geq0}$ be a symmetric L\'evy process with intensity
$\nu(x)
\asymp e^{-a|x|^{\beta}} |x|^{-d-\delta}
(1+|x|)^{d+\delta-\gamma}$, where $a > 0$, $\beta\in(0,1]$, $\delta
\in[0,2)$ and $\gamma> (d+1)/2$.
Such processes were considered in more general settings in \cite{bibCKK} (see also~\cite{bibCKu3}
and references therein). As discussed in Example~\ref{exex1}(3),
Assumption~\ref{assassnu} is verified.
Moreover, as proved in a greater generality in \cite{bibCKK}, Theorem 1.2(1), $(X_t)_{t \geq0}$ is a strong
Feller process with bounded continuous transition densities satisfying
appropriate sharp two-sided
bounds with respect to large and small times separately (see \cite{bibCKK}, (1.13) and (1.14)). Thus,
also Assumption~\ref{assassdensity} holds. As before, the required
bound on the Green function in
Assumption~\ref{assassbhi} may be obtained by showing the same
estimate for the $\lambda$-potential
kernel of the process $(X_t)_{t \geq0}$, which can be easily done by
using the
sharp transition density
estimates referred to above. This class includes a large family of
(exponentially) tempered symmetric
stable processes \cite{bibRos} [$a>0$, $\beta= 1$, $\gamma= \delta+
d$, $\delta\in(0,2)$]
and the rotationally symmetric relativistic stable processes above.
\end{example}

We also give two examples of processes, which do not satisfy some of
our assumptions.

\begin{example}\label{exneg2}
(1)~\emph{Rotationally\vspace*{1pt} symmetric geometric $2$-stable} (\emph{gamma variance}) \emph{process}.
Let $\psi(\xi)= \log(1+|\xi|^2)$. In\vspace*{1pt} this case the L\'evy intensity is $\nu(x) \asymp
|x|^{-d}e^{-|x|}(1+|x|)^{(d-1)/2}$. As in
Example~\ref{exneg1}, at least in dimension one, the L\'evy measure
does not satisfy condition (3) of
Assumption~\ref{assassnu}.

(2)~\emph{Iterated rotationally symmetric geometric $\alpha$-stable process}.
Let $\psi(\xi)=\log(1+\log^{\alpha}(1+|\xi|^{\alpha}))$,
$0<\alpha<2$.
It can be checked directly that
the transition densities are unbounded for any $t>0$ (see \cite{bibBBKRSV}, page~117) and the second part
of Assumption~\ref{assassdensity} fails.
\end{example}

\subsection{Decay of ground state and intrinsic ultracontractivity-type properties}\label{sec4.2}
\label{subsecexres}

It is useful to see how Theorem~\ref{thmgsest} translates to
particular cases of processes. In the
following, we give explicit examples of ground state decays and compare
our results with others.

\begin{example}\label{exex6}
(1)~\emph{Rotationally symmetric non-Gaussian stable and related processes}
discussed in Examples
\ref{exex2}(1)--(4) and~\ref{exex4} above. In particular, this
includes mixtures of two stable
processes with different stability indices, jump-diffusion,
rotationally symmetric geometric
$\alpha$-stable [with $\alpha\in(0,2)$], and \mbox{symmetric} stable-like
L\'evy processes. In this
case, we have $\nu(x) \asymp|x|^{-d-\alpha}$, $|x| > 1/2$, $\alpha
\in
(0,2)$, and hence
\[
\varphi_0(x) \asymp G^V_{B(x,1)}\mathbf{1}(x)
|x|^{-d-\alpha}, \qquad|x| > R.
\]
Furthermore, when also the condition in Corollary~\ref{comparable} is
satisfied, then
\[
\varphi_0(x) \asymp\frac{1}{(1+|x|)^{d+\alpha}(1+V_{+}(x))}, \qquad x \in
\mathbf{R}^d.
\]
This clearly recovers the results for non-Gaussian symmetric stable
processes in~\cite{bibKaKu,bibKL}.

(2)~\emph{Symmetric L\'evy process with nondegenerate Brownian part}.
For this, see Example~\ref{exex3} above. In this case, Theorems~\ref{thmbsest} and~\ref{thmgsest} (also
Corollary~\ref{comparable}) allow to identify the leading order of
decay of the ground state at infinity
(as well as provide upper bounds for higher order eigenfunctions) as
the contribution of the L\'evy intensity
$\nu$ and a correction from the potential $V$. However, since our
constants are not optimal, it can be expected
that the correct asymptotics should contain a further term of smaller
order (similar to Carmona's bound in
\cite{bibCa}) coming from the Brownian component of the process.
However, showing this requires a more subtle
argument and cannot be seen from our present results.

(3)~\emph{Symmetric jump L\'evy processes with exponentially localized L\'
evy measure}.
See Example~\ref{exex5} above. It is a well-known result in \cite{bibCMS},\vspace*{1pt} Proposition IV.4, that if
$e^{-t\psi(\cdot)} \in L^1(\mathbf{R}^d)$, $t>0$, and there is $b>0$
such that
$\int_{|x|>1} e^{b|x|} \nu(dx) < \infty$,
then $|\varphi_n(x)|\leq C e^{-C'|x|}$, $x \in\mathbf{R}^d$, with
$C=C(X,V,n)$, $C'=C'(X,V,n)$, that is, if the L\'evy
measure is exponentially localized, then the fall-off of the
corresponding eigenfunctions is also exponential.
Note that Theorems~\ref{thmbsest}--\ref{thmgsest} (also Corollary~\ref{comparable}) essentially improve this
result under assumptions which are not more significantly restrictive
than those of \cite{bibCMS}.

(4)~\emph{Rotationally symmetric relativistic stable process}.
Compare Example~\ref{exex2}(5). This is a special case of the class
discussed in (3) above. It was proven in
\cite{bibKS}, Theorem 1.6, that if $V$ is a nonnegative, locally
bounded potential comparable to a rotationally
symmetric function, radially nondecreasing and comparable on unit balls
with $\lim_{|x| \to\infty}V(x)/|x|=\infty$
(i.e., the corresponding Feynman--Kac semigroup is IUC), then
\[
\varphi_0(x) \asymp\frac{e^{-m^{1/\alpha}|x|}}{(1+|x|)^{(d+\alpha
+1)/2}(1+V(x))}, \qquad x \in
\mathbf{R}^d.
\]
Theorem~\ref{thmgsest} above generalizes this result to the
substantially larger space of \mbox{$X$-}Kato class
potentials with no restrictions on the order of growth of the potential
at infinity and with no use being
made of intrinsic ultracontractivity properties. We obtain the result
of \cite{bibKS} as the second part
of Corollary~\ref{comparable}. Note also that Theorem~\ref{thmbsest}
is completely new in this context.

(5)~\emph{Diffusions}. It is useful to compare our results to the classic
facts known for the eigenfunctions of
Feynman--Kac semigroups involving Brownian motion (i.e., Schr\"odinger
semigroups generated by $-\Delta+V$).
In the general case, $t_0$-GSD and $t_0$-IUC (see the definitons in
Section~\ref{subsecIUC}) imply directly that for all $x \in\mathbf
{R}^d$ and~$n \geq1$
%
\begin{eqnarray}\label{eqdomi}
\bigl|\varphi_n(x)\bigr| &\leq& C(t_0) \llVert
\varphi_n\rrVert _{\infty} e^{(\lambda
_n-\lambda_0)t_0} \varphi_0(x) \quad \mbox{and}
\nonumber\\[-8pt]\\[-8pt]
\bigl|\varphi_n(x)\bigr| &\leq& C(t_0) e^{(\lambda_n-\lambda
_0)t_0} \varphi_0(x),\nonumber
\end{eqnarray}
respectively. In Corollary~\ref{cordom}, we get without any use of
AGSD/AIUC-type properties that for a large class of jump L\'evy processes
%
\begin{equation}
\label{eqdomi2} \bigl|\varphi_n(x)\bigr| \leq C(X,V,n) \varphi_0(x),
\qquad x \in\mathbf {R}^d, n \geq1
\end{equation}
with a constant $C(X,V,n)$. In comparison with (\ref{eqdomi}), the
dependence on $n$ of the constant $C(X,V,n)$ is rather implicit, but
(\ref{eqdomi2}) still says that for each fixed $n$ the~decay of
$\varphi_n$ at infinity is dominated by that of $\varphi_0$. This
markedly differs from the diffusive case, where the estimates as in
(\ref{eqdomi2}) cannot be taken for granted in lack of AGSD/AIUC-type
properties. This can be seen, for instance, in the
example of the harmonic oscillator, for which $V(x)=|x|^2$ and the
eigenfunctions are given by the Hermite
functions. It is known that in this case the semigroup is not IUC (not
even AIUC). A direct analysis shows that (\ref{eqdomi2})
does not occur either.
\end{example}

We now illustrate our results on intrinsic ultracontractivity-type
properties by specific examples.

\begin{example}
Theorems~\ref{thmsuffiuc}--\ref{thmneciuc} and Propositions~\ref{propAIUCproba}--\ref{propIUCproba}
apply directly to the following three classes of examples with
different growth rate of borderline
potentials. For any of these specific examples, Assumption~\ref{assasspnu} can be verified by using
the time--space estimates of the related transition densities, as in
Examples~\ref{exex2}--\ref{exex5}.
\begin{longlist}[(2)]
\item[(1)]
\emph{Borderline potentials of logarithmic order}. The borderline behavior
\[
-\log\nu(x) \asymp-\log\mathbf{P}^x\bigl(X_t \in B(0,1)
\bigr) \asymp-\log p(t,x) \asymp\log|x|
\]
occurs in the following cases:
\begin{enumerate}[(a)]
\item[(a)]
\emph{Jump stable-type processes with bounded transition densities}
[see Examples~\ref{exex2}(1)--(3) and~\ref{exex4}]: this includes non-Gaussian rotationally
symmetric stable processes (our results
recover and substantially improve the methods of \cite{bibKaKu,bibKL}), mixtures of rotationally symmetric
stable processes, jump-diffusions [in this case the Brownian component
has no effect on (A)GSD and (A)IUC],
and symmetric stable-like L\'evy processes. In this case, Theorem~\ref{thmiucvelgsd} implies equivalence of
GSD and IUC.

\item[(b)]
\emph{Rotationally symmetric geometric $\alpha$-stable processes,
$\alpha\in(0,2)$} [see Example~\ref{exex2}(4)]: by Proposition~\ref{propultracontractivity} the semigroup $\{
T_t\dvtx  t \geq0\}$ is not IUC (i.e., it is not $t$-IUC
for small $t>0$, not even ultracontractive), while it is GSD and
$t$-IUC for $t>2 t_b$ provided the potential
$V$ is pinning enough.
\end{enumerate}

\item[(2)]
\emph{Borderline potentials of linear order}. The borderline behavior
\[
-\log\nu(x) \asymp-\log\mathbf{P}^x\bigl(X_t \in B(0,1)
\bigr) \asymp-\log p(t,x) \asymp|x|
\]
occurs for rotationally symmetric L\'evy processes satisfying our
assumptions provided that their L\'evy
intensities are exponentially decaying at infinity (the case $\beta=1$
in Example~\ref{exex5}).
Important examples to this class are rotationally symmetric
relativistic stable processes and tempered rotationally symmetric
stable processes. For relativistic
stable processes, it was proven in \cite{bibKS} that when $V$ is a
nonnegative, locally bounded potential
comparable to a function which is rotationally symmetric, radially
nondecreasing and comparable on unit balls,
then the corresponding Feynman--Kac semigroup is IUC if and only if
$\lim_{|x| \to\infty}V(x)/|x|=\infty$.
The combination of Theorems~\ref{thmiucvelgsd} and~\ref{thmneciuc}
generalizes this result to the
substantially larger class of $X$-Kato class potentials. Also, we
obtain the result of \cite{bibKS} in
Corollary~\ref{corborder}(1) under Assumption~\ref{assasscomp}.

\item[(3)]
\emph{Borderline potentials of sublinear but faster than logarithmic
order}. The borderline behavior
\[
-\log\nu(x) \asymp-\log\mathbf{P}^x\bigl(X_t \in B(0,1)
\bigr) \asymp-\log p(t,x) \asymp|x|^{\beta}, \qquad0 < \beta< 1,
\]
appears in the case of processes with L\'evy measures decaying
subexponentially at infinity
(the case $\beta<1$ in Example~\ref{exex5}).
\end{longlist}

Note that, roughly speaking, this is the complete range of possible
borderline growths for the processes
we consider. An asymptotic growth of the order $|x|^{\beta}$ with
$\beta
>1$ is ruled out by Assumption
\ref{assassnu}(3), see the discussion in Example~\ref{exneg1}. Also,
the borderline potential cannot
be slower than logarithmic due to the integrability of the L\'evy
intensity $\nu$ outside a neighborhood
of the origin. We also note that linear growth is the quickest possible
as well for the class of subordinate
Brownian motions obtained under subordinators whose L\'evy exponents
are complete Bernstein functions, see
\cite{bibKSV}, Lemma 2.1.
\end{example}

A second type of example is about Feynman--Kac (in fact, Schr\"odinger)
semigroups involving standard Brownian
motion under a potential. Although the strictly diffusive case when
$\nu\equiv0$ is not covered by our paper,
it is interesting to compare the results to better understand what
mechanism lies behind IUC in the general case.

\begin{example}\label{whyeasier}
\emph{Diffusions}. In the classic papers on IUC of Schr\"odinger
semigroups generated by $-\Delta+V$, it was
considered whether the property holds for some special ways of choosing
the potential. In the one-dimensional
case \cite{bibDS}, Theorem 6.1, shows that when $V(x) = |x|^a$, $a>0$,
or $V(x)=|x|^2\log(|x|+2)^b$, $b>0$,
then the related semigroup is IUC if and only if $a>2$ and $b>2$ (i.e.,
fails for $a, b \leq2$). When $d
\geq1$, it is shown in \cite{bibDS}, Theorem 6.3, that IUC occurs
whenever $C_1 + C_2 |x|^b \leq V(x) \leq
C_3 + C_4 |x|^a$, with $a/2 +1 <b$. To the best of our knowledge, AIUC
was not considered before the paper
\cite{bibKL}. However, by a use of the Mehler formula it follows that
in the case of the harmonic oscillator
AIUC does not occur, see \cite{bibD}, Theorem 4.3.2. Recently, in \cite{bibAT} a general sufficient condition
for IUC was found for Schr\"odinger semigroups. For radial potentials
$V$, this condition is also necessary
and it is formulated as
%
\begin{equation}
\label{eqIUCclasscond} \int_{r_0}^{\infty} \frac{1}{\sqrt{V(r)}}
\,dr < \infty\qquad\mbox{for some } r_0 >0.
\end{equation}
For instance, for the potential
\begin{eqnarray}
V(x)=|x|^2 \bigl(\log|x|\bigr)^2 \bigl(\log\log|x|\bigr)^2\cdots \bigl(
\underbrace{\log \cdots\log |x|}_{(m-1)\mbox{-}\mathrm{times}}\bigr)^2 \bigl(\underbrace{\log\cdots
\log|x|}_{m\mbox{-}\mathrm{times}}\bigr)^{2+\delta},\nonumber
\\
\eqntext{m \in\mathbf{N}, \delta\geq0,}
\end{eqnarray}
this condition is satisfied if and only if $\delta>0$. This means that
IUC holds for an arbitrary choice of
$m \in\mathbf{N}$ whenever $\delta>0$, and suggests that in the
diffusion case
it is not possible to identify the
borderline potential directly as in the case of jump processes. Note
that all of the classic results discussed
above were obtained by purely analytic arguments. We believe that it is
possible to derive the analytic
condition (\ref{eqIUCclasscond}) by probabilistic methods based on
sufficiently efficient estimates of the
expression at the right-hand side of (\ref{eqprobchar}). For instance,
when $V \geq0$ satisfies Assumption
\ref{assasscomp} and the semigroup is IUC, then a rough estimate gives
$\lim_{|x| \to\infty} V(x)/|\log
\mathbf{P}^x(X_t \in D)|= \infty$, allowing correctly to identify
$|x|^2$ as
the leading order of borderline growth,
as in Proposition~\ref{propIUCproba}(1).
\end{example}

In the context of diffusions, we also mention that a condition similar
to part (3) of Assumption~\ref{assassnu}
has been used for Green functions of elliptic differential operators on
domains in \cite{bibMur07,bibMM09,bibTom}
and related papers, and it goes back to \cite{bibPin89,bibPin99},
where it was introduced as a small-perturbation
condition of an elliptic operator by another operator. In particular,
it is shown that intrinsic ultracontractivity
implies the small-perturbation condition.


\section*{Acknowledgments}
The authors thank IHES, Bures-sur-Yvette,
for a visiting fellowship providing an ideal environment to joint work.
We are pleased to thank \mbox{T.~Kulczycki} for discussions on intrinsic ultracontractivity
and eigenfunction estimates. Our
special appreciation goes to M. Kwa\'snicki for reading a preliminary
version of the manuscript, and for
many valuable comments and discussions on intrinsic ultracontractivity
and estimates of
local extrema of harmonic functions. We also thank the anonymous
referee for a careful reading
of the manuscript and observations which improved our paper.



\printaddresses

\end{document}